\documentclass[reqno]{amsart}

\usepackage{lmodern}
\usepackage[T1]{fontenc}
\usepackage[utf8]{inputenc}
\usepackage[english]{babel}

\usepackage{microtype,amsmath,amssymb,amsthm,mathrsfs,cases,
latexsym,mathtools,mathdots,booktabs,xparse,enumitem,tikz,bm,url}
\usepackage[centertableaux]{ytableau}

\usetikzlibrary{arrows,positioning,shapes.geometric}
\usepackage[pdftex]{hyperref}
\usepackage[capitalize,noabbrev]{cleveref}

\usepackage{todonotes}
\presetkeys%
    {todonotes}%
    {inline,backgroundcolor=yellow}{}

\newtheorem{theorem}{Theorem}
\newtheorem{proposition}[theorem]{Proposition}
\newtheorem{lemma}[theorem]{Lemma}
\newtheorem{corollary}[theorem]{Corollary}
\newtheorem{conjecture}[theorem]{Conjecture}

\theoremstyle{definition}
\newtheorem{definition}[theorem]{Definition}
\newtheorem{example}[theorem]{Example}
\newtheorem{remark}[theorem]{Remark}

\definecolor{lightblue}{rgb}{0.8,0.8,1.0}
\definecolor{lightgreen}{rgb}{0.8,1.0,0.8}
\definecolor{pBlue}{RGB}{120,160,230}
\definecolor{pCyan}{RGB}{149,186,201}
\definecolor{pSand}{RGB}{184,166,121}
\definecolor{pAlgae}{RGB}{87,115,135}
\definecolor{pSkin}{RGB}{236,216,167}
\definecolor{pGray}{RGB}{156,175,156}
\definecolor{pPink}{RGB}{215,114,127}
\definecolor{pOrange}{RGB}{211,153,80}
\definecolor{coral}{RGB}{252, 53, 166}
\definecolor{magenta}{RGB}{94, 30, 107}
\definecolor{teal}{RGB}{188, 244, 247}
\definecolor{pea}{RGB}{146, 252, 148}

\DeclareMathOperator{\pair}{pair}
\DeclareMathOperator{\boun}{bd}
\DeclareMathOperator{\invar}{pr}

\DeclareMathOperator{\height}{ht}
\DeclareMathOperator{\width}{width}

\newcommand{\bw}[1]{\mathtt{#1}} % Binary words

\NewDocumentCommand\qbinom{O{q} m m o}{
\ensuremath{\IfNoValueTF{#4}{
{\genfrac{[}{]}{0pt}{}{#2}{#3}_{#1}}
}{
{\genfrac{[}{]}{0pt}{}{#2}{#3}_{#1}^{#4}}
}}}

\newcommand{\defin}[1]{%
\relax\ifmmode%
\textcolor{blue}{#1}%
\else\textcolor{blue}{\emph{#1}}%
\fi%
}

\newcommand{\thsup}{\textnormal{th}}

\newcommand{\bit}{\mathtt{b}}
\newcommand{\setN}{\mathbb{N}}
\newcommand{\setR}{\mathbb{R}}
\newcommand{\setZ}{\mathbb{Z}}
\newcommand{\setQ}{\mathbb{Q}}
\newcommand{\setP}{\mathbb{N}_{>0}} %Positive integers

\newcommand{\avec}{\mathbf{a}}
\newcommand{\bvec}{\mathbf{b}}
\newcommand{\xvec}{\mathbf{x}}
\newcommand{\yvec}{\mathbf{y}}
\newcommand{\uvec}{\mathbf{u}}
\newcommand{\vvec}{\mathbf{v}}
\newcommand{\wvec}{\mathbf{w}}

\newcommand{\RT}{\mathrm{RT}} %Ribbon tableaux
\newcommand{\SSYT}{\mathrm{SSYT}}
\newcommand{\CSSYT}{\mathrm{CSSYT}}

\newcommand{\monomial}{\mathrm{m}}
\newcommand{\qmonom}{\mathrm{M}}
\newcommand{\schurS}{\mathrm{s}}
\newcommand{\completeH}{\mathrm{h}}
\newcommand{\powersum}{\mathrm{p}}

\newcommand{\symS}{\mathfrak{S}}

\newcommand{\rel}{\mathcal{R}}	%Poset relations
\newcommand*{\orderPresF}{\mathcal{A}} %Weakly order-preserving functions
\newcommand*{\equalOrderPresF}{\mathcal{A}^{=}}

\newcommand{\poseta}{a}
\newcommand{\posetb}{b}
\newcommand*{\equalK}[1]{K^{=}_{#1}}
\newcommand*{\strictK}[1]{K^{<}_{#1}}

\newcommand{\opsurj}{\mathcal{O}} %Order-preserving surjections
\newcommand*{\erel}[1]{{\tilde{#1}}} %Chain congruence.
\newcommand*{\mecs}[2]{m_{#1}(#2)} % Size of equivalence class containing the unique minima

\DeclareMathOperator{\length}{\ell}

\tikzset{every picture/.append
  style={scale=1,
	baseline=(current bounding box.center),
	x=1em,
	y=1em,
	thinLine/.style={line width=0.7pt},
	thickLine/.style={line width=1.4pt,double,line join=round},
	entries/.style={xshift=-0.5em,yshift=-0.5em,font=\small},
	circled/.style={circle,draw,inner sep=1pt,minimum size=1.4em}
	}
}

\title{Cylindric Schur functions}

\author{Per Alexandersson}
\address{Department of Mathematics, Stockholm University, SE-106 91 Stockholm, Sweden}
\email{per.w.alexandersson@gmail.com}

\author{Ezgi Kantarci Oğuz}
\address{Department of Mathematics, Galatasaray University, İstanbul, Turkey}
\email{ezgikantarcioguz@gmail.com}

\begin{document}
\begin{abstract}
We generalize several classical results about Schur functions
to the family of cylindric Schur functions.
First, we give a combinatorial proof of a Murnaghan--Nakayama
formula for expanding cylindric Schur functions in
the power-sum basis.
We also explore some cases where this formula is cancellation-free.

The second result is polynomiality of Kostka coefficients
associated with stretched row-flagged
skew Schur functions. This implies polynomiality
of stretched cylindric Kostka coefficients.
This generalizes a result by E.~Rassart from 2004.

Finally, we also show the saturation property for
the row-flagged skew Kostka coefficients
which also implies the saturation property for cylindric
Schur functions.
\end{abstract}

\maketitle

\setcounter{tocdepth}{1}
\tableofcontents

\section{Introduction}

Cylindric Schur functions (defined further down in \cref{def:cylindricSkewSchur}) 
are a family of symmetric functions that were introduced by A.~Postnikov in \cite{Postnikov2005},
and generalize the classical Schur functions.
The main application of cylindric Schur functions is their
connection with Gromow--Witten invariants.
Postnikov conjectured that some of these cylindric Schur functions
(so called \emph{toric Schur polynomials}) are the Frobenius 
characteristic of certain \emph{toric Specht modules}.
The characters of these modules are given as the
coefficients of the power-sum expansion of the toric Specht modules and
it is therefore of interest to find a combinatorial formula for these.

\subsection{The cylindric Murnaghan--Nakayama rule}

By definition, characters of toric Specht modules generalize the classical
irreducible characters for the symmetric group, $\chi^{\lambda}$.
Such irreducible characters can be computed via the
Murnaghan--Nakayama rule \cite{Murnaghan1937,Nakayama1940}.
It is therefore natural to ask for a generalization of this 
rule in the more general context of cylindric Schur functions.
In this paper, we give a combinatorial proof of such a formula in \cref{thm:cylindricMNRule}.
A similar formula seems to appear in \cite[Eq. (145)]{KorffPalazzo2020}.

\subsection{Cancellations in the Murnaghan--Nakayama rule}

The Murnaghan--Nakayama rule for computing $\chi^{\lambda}(\mu)$ is a signed sum.
However, when all parts of $\mu$ are of the same size, this sum is cancellation-free.
One may ask if the cylindric Murnaghan--Nakayama rule is cancellation-free for such
shapes, but this is not true in general.
However, when the inner boundary of the cylindric diagram is
a staircase, the Murnaghan--Nakayama rule is
cancellation-free, see \cref{cor:cancellationFree}.
Note that most of the existing literature on tilings of regions is focused
on simply connected regions and does not apply
to cylindric shapes as they are not simply connected in general.

\subsection{Flagged and cylindric Gelfand--Tsetlin patterns}

In the last few sections of the paper, we explore Gelfand--Tsetlin
patterns arising from flagged skew Schur functions and
cylindric skew Schur functions, respectively.
These patterns (certain arrays of integers)
are in natural bijection with flagged (or cylindric) semi-standard Young tableaux.
One can view the Gelfand--Tsetlin patterns as lattice points in
so-called Gelfand--Tsetlin polytopes.

Each cylindric skew Schur function can be expressed as a non-negative sum of
flagged skew Schur functions in a natural manner.
Hence, several natural properties of flagged skew Schur functions
carry over to cylindic Schur functions.

Two such properties are \emph{saturation of Kostka coefficients}
and \emph{polynomiality of stretched Kostka coefficients},
which we prove in \cref{cor:flaggedKostkaPolynomial} and \cref{cor:cylindricKostkaPolynomial}.
These properties are closely related to properties of the corresponding
Gelfand--Tsetlin polytopes. The first property
is implied by the fact that every non-empty (flagged or cylindric)
Gelfand--Tsetlin polytope contains a lattice point,
see \cref{thm:flagSaturation,thm:cylindricSaturation}.
The second property is exactly period collapse of the
corresponding Ehrhart (quasi)polynomial.

\section{Preliminaries}

\subsection{Diagrams}\label{sec:cylindricSkew}

For an integer partition $\lambda_1 \geq \lambda_2 \geq \dotsb \geq \lambda_\ell$,
the corresponding \defin{Young diagram} is the set of coordinates
$\{(i,j) : 1\leq i \leq \lambda_j, 1 \leq j \leq \ell \} \subseteq \setZ^2$.
We call these coordinates \defin{boxes} and display them as so (French convention).
If $\mu$ is another partition with $\mu_i \leq \lambda_i$
for all $i$, we write $\mu \subseteq \lambda$
and we let $\lambda/\mu$ denote the set of boxes in $\lambda$
but not in $\mu$.
For example, the diagrams $(4,4,3,1)$
and $(4,4,3,1)/(2,2)$ are rendered as follows:
\[
\ytableausetup{boxsize=0.8em}
 \ydiagram{1,1,3,4,4}
 \qquad
  \ydiagram{1,1,3,2+2,2+2}
\]
\medskip

We may also describe a Young diagram by describing its outer boundary.
We let $\bw{0}$'s indicate right steps and $\bw{1}$'s correspond to down steps.
For example, the outer boundary of the partition $64442$ is 
described by the binary word $\bw{\dotsc 111 \; 00100.111001 \; 000 \dotsc}$ as seen below:
\begin{center}
\includegraphics[scale=1,page=1]{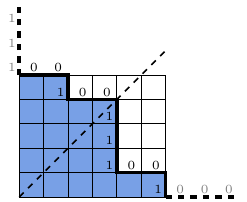}
\end{center}
We mark the position in the word where the boundary path passes the line $x=y$.

\subsection{Cylindric diagrams}

We now consider Young diagrams drawn on a cylinder,
where the top and bottom (or left and right side) 
of the diagram are glued together. 
Let $\mathfrak{C}_{x,y}$ denote the cylindric lattice structure 
given by identifying $(0,0)$ and $(x,-y)$ in $\mathbb{Z}^2$.
A lattice path from $(0,0)$ to  $(x,-y)$ in $\mathbb{Z}^2$, or any translate of 
it becomes a single loop in $\mathfrak{C}_{x,y}$, which we call 
 a \defin{boundary loop}.

A \defin{cylindric diagram} $D$ is a pair of boundary loops, $B_I$ and $B_O$
such that $B_I$ lies weakly left/below of $B_O$. We may assume that the starting point of $B_I$ is $(0,0)$ and
the starting point of $B_O$ is some point $(l,l)$ with $l\geq 0$.
We refer to $B_I$ as the \defin{inner boundary} and $B_O$ as the \defin{outer boundary}.
One can view a boundary loop as a binary word with $x$ $\bw{0}$'s
$y$ $\bw{1}$'s together with a starting point in $\mathfrak{C}_{x,y}$.
We use the convention of indexing the entries of boundary words modulo $x+y$, 
starting with the entry following the $x=y$ line.

\begin{example}
In Figure~\ref{fig:cylindricDiagramExample}, we see an example of a 
cylindric diagram in $\mathfrak{C}_{6,5}$. 
The inner boundary starts at $(0,0)$ and has boundary word $\wvec_I=\bw{00101101100}$.
The outer boundary starts at $(4,4)$ and has boundary word $\wvec_O=\bw{10001001110}$.
\begin{figure}[!ht]
\begin{center}
\includegraphics[scale=0.8,page=2]{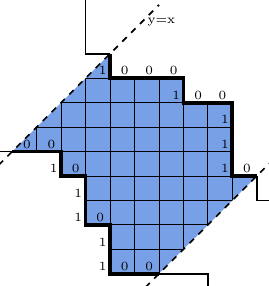}
\end{center}
\caption{An inner and outer boundary loop for $\mathfrak{C}_{6,5}$.}\label{fig:cylindricDiagramExample}
\end{figure}

\end{example}

Unlike in the non-cylindrical setting, the starting points are necessary, as the boundary words
alone do not uniquely determine the shape.  In fact, there are infinitely many diagrams that share
the same inner and outer boundary.

\begin{figure}[!ht]
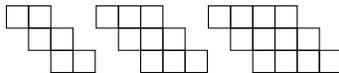

\centering
\ytableausetup{smalltableaux}
\ydiagram{2,1+2,2+2} \ydiagram{3,1+3,2+3} \ydiagram{4,1+4,2+4}
    \caption{All the given diagrams have both inner and outer boundary words of form $\bw{101010}$ or $\bw{010101}$ .}
    \label{fig:boundary}
\end{figure}

We identify a cylindrical diagram $D$ with the set of unit boxes in 
rows $-1$ to $-h$ that lie between the two boundary lines. 
The boxes form a skew diagram $\lambda/\mu$ that we refer to as the \defin{shape} of $D$. 
Note that any Young diagram can be realized as a cylindric diagram with some empty rows.

Though our definition is coordinate based, essentially all
of the properties we  work with in this paper
are translation independent.
When the starting point is not important, we depict the cylindric diagrams by
only drawing their shape, using dots to show
where bottom row and top row are glued.

For example, the cylindric diagram $D$  from Figure ~\ref{fig:cylindricDiagramExample} can alternatively be described 
either by taking five consecutive rows and indicating 
where the last row should attach to the first row.
A similar constriction works for columns:
\begin{equation}
 \ytableausetup{boxsize=0.5em}
\begin{ytableau}[*(pBlue)]
\none[\cdot] & \none[\cdot]& \none[\cdot]& \none[\cdot]& \none[\cdot]& \none[\cdot]& \none[\cdot]& \none[\cdot]& \none[\cdot]& \none[\cdot]& \none[\cdot]\\
\none & \none & \none & \none & & & &&&& \\ 
\none & \none & \none & \none &\none & &&&&&& \\ 
\none & \none & \none & \none &\none & &&&&&& &&& \\
\none & \none & \none & \none &\none & \none & &&&&&& &&& & \\
\none & \none & \none & \none & \none & \none & \cdot & \cdot & \cdot & \cdot & \cdot & \cdot & \cdot &\cdot & \cdot & \cdot & \cdot \\
 \end{ytableau} 
 \qquad 
 \begin{ytableau}[*(pBlue)]
\cdot \\
\cdot & & \\
\cdot & & \\
\cdot & & \\
\cdot & & & \\
\cdot & & & & & & \none[\cdot]\\
\cdot & & & & & & \none[\cdot]\\
\cdot & & & & & & \none[\cdot]\\
\none & \none & & & & & \none[\cdot]\\
\none & \none & \none & & & & \none[\cdot]\\
\none & \none & \none & & & & \none[\cdot]\\
\none & \none & \none & \none & & & \none[\cdot]\\
\none & \none & \none & \none & & & \none[\cdot]\\
 \end{ytableau}
\end{equation}
The reader is encouraged to verify that these indeed describe 
the same diagram.

\subsection{Ribbon diagrams in the non-cylindrical setting}

In this section we will give a very brief introduction to how ribbon diagrams work in 
the non-cylindrical case and give an overview of some key properties that we will use to 
contrast with their cylindrical analogue. 

A connected skew-diagram of $k$ boxes that contains no $2{\times}2$ block is called a \defin{$k$-ribbon}. 
The \defin{height} of a $k$-ribbon is given by the number of rows it intersects with minus $1$. 
A (possibly skew) Young diagram that contains no removable $k$-ribbon is called a \defin{k-core}.
For a given Young diagram $\lambda$, a $k$-ribbon filling of $\lambda$ is a 
sequence $\lambda_0\subset\lambda_1\subset\cdots\lambda_s=\lambda$ 
where $\lambda_0$ is the $k$-core and each $\lambda_i/\lambda_{i-1}$ is a $k$-ribbon. 
The partition $\lambda_0$ is independent of the choice of the ribbons and is called the \defin{$k$-core} of $\lambda$.

For a given $\lambda$ with boundary word $w$ and no $k$-core, the \defin{Littlewood quotient map}
sends $\lambda$ to the $k$-tuple $(\lambda_1,\lambda_2,\ldots,\lambda_{k})$ where $\lambda_i$ 
is given by the letters of $w$ whose indices equal $i$ modulo $k$ (We the convention that the 
index $0$ is given to the edge to the right of $x=y$ line). If the $k$-core is non-empty, 
we get the Littlewood quotient by using the same algorithm for the core, and 
skewing the $i^\thsup$ part of $\lambda$ by the $i^\thsup$ part of the core. 
The resulting $k$-tuple is called the \defin{$k$-quotient} of $\lambda$.

For the example $\lambda=(4,4,3,1)$ with boundary word $\bw{\ldots111  01001011  000\ldots}$, 
taking $k=4$ gives us the four words  $\bw{\ldots111000\ldots}$, $\bw{\ldots 11101000\ldots}$, 
$\bw{\ldots 111001000\ldots}$,  $\bw{\ldots 11101000\ldots}$. 

\[
\ytableausetup{boxsize=0.8em}
 \ydiagram{1,3,4,4}
 \qquad \rightarrow \qquad
  \left( \, \ydiagram{1} , \quad \varnothing, \quad \ydiagram{1}, \quad \ydiagram{1} \, \right)
\]

The Littlewood map gives a bijection between $k$-ribbon fillings of $\lambda$ and ``standard fillings'' of its quotient, 
fillings where each number from $1$ to the $(|\lambda|-|\lambda_0|)/k$ is used exactly once. 
In particular, removing a $k$-ribbon from the outer boundary corresponds to exchanging a $1$ in 
some position $w_{t+k}$ with a $0$ in position $w_t$ on the boundary word, which in turn corresponds to 
removing a box from $\lambda_i$, where $i\equiv t$ modulo $k$. 
Here, the bottom right box of the removed ribbon falls on the diagonal $x=y+t+k$, 
so $t+k$ is also called the \defin{diagonal value} of the ribbon and gives us a way to 
easily match ribbons to boxes in the quotient without looking at boundary words.

One can reach from any $k$-filling to any $k$-filling by doing pairwise \emph{flips}, 
which correspond to exchanging $t \leftrightarrow t+1$ in the $k$-core. 
The flipping operation always changes the total height by a multiple of two, so 
the parity of the total of the heights of $k$-ribbons is the same for each filling.

% \Ezgi{example of some fillings}

\begin{definition}[Comparison with cylindric ribbon diagrams]

A cylindric diagram on $\mathfrak{C}_{x,y}$ with $k\leq x+y$ boxes is called a \defin{$k$-ribbon}
if it is connected with no $2{\times}2$ block of boxes. 
A ribbon containing exactly $x+y$ boxes is called a \defin{loop ribbon}. 
The \defin{height} of a ribbon is given by the number of rows it intersects $-1$, except for 
loop ribbons. All loop ribbons have height $y$ (One can think of the height as counting pairs 
of consecutive rows, giving us the number of rows $-1$.
In the case of loops, the top and bottom
rows are also consecutive, so the height is the number of rows it intersects, which are all the rows of the shape).

A $k$-ribbon $R$ on a cylindric shape $D$ is called \defin{removable} if it 
has the same outer boundary as $R$ and $D/R$ is a cylindric diagram. 
A cylindric diagram with no removable $k$-ribbon is said to be a \defin{$k$-core}.

A cylindric $k$-ribbon diagram is a list of cylindric diagrams $D_0\subset D_1\subset\cdots\subset D_s=D$ 
where each $D_i/D_{i-1}$ is a removable $k$-ribbon and $D_0$ is a $k$-core.
\end{definition}

% \Ezgi{Add an example}

Ribbons in the cylindric setting lack many of the properties of their non-cylindric 
analogs which makes them tricky to work with. 
Firstly, we do not necessarily have a unique $k$-core.

\begin{itemize}
 \item The core may not have a fixed position.
  Here are two fillings with different $4$-cores:
 \begin{equation}\label{eq:core1}
 \begin{ytableau}
*(pSkin) \cdot &  *(pSkin) \cdot \\
\none &  &*(pBlue) \\
\none &\none & *(pBlue) & *(pBlue) & *(pSkin) \\
\none & \none & \none & *(pBlue)  & *(pSkin)\\
 \none & \none & \none &  \none & \none[\cdot] & \none[\cdot]  
\end{ytableau}
\qquad 
\begin{ytableau}
*(pSkin) \cdot &  *(pSkin) \cdot \\
\none & *(pSkin) & *(pBlue) \\
\none &\none & *(pBlue) & *(pBlue) & *(pBlue) \\
\none & \none & \none &  & *(pSkin)\\
\none & \none & \none &  \none & \none[\cdot] & \none[\cdot]  
\end{ytableau}
\end{equation}

 \item The core may not have a fixed shape.
Here are two fillings with different $5$-core shapes.
\begin{equation}\label{eq:core2}
 \begin{ytableau}
 \cdot & *(pBlue)  \cdot & *(pBlue)  \cdot \\
 \none &  & *(pBlue)   & *(pBlue)  \\
\none &\none &  & *(pBlue) & *(pPink) \\
\none & \none & \none &   & *(pPink)  & *(pPink)\\
\none & \none & \none & \none &   & *(pPink)  & *(pPink)\\
\none & \none & \none &  \none & \none& \none[\cdot] & \none[\cdot]& \none[\cdot] 
\end{ytableau}
\qquad 
 \begin{ytableau}
  \cdot & \cdot & *(pBlue)  \cdot \\
\none &  &   *(pBlue)  & *(pBlue)  \\
\none &\none &  & *(pBlue) & *(pPink) \\
\none & \none & \none & *(pBlue)  & *(pPink)  & *(pPink)\\
\none & \none & \none & \none &   & *(pPink)  & *(pPink)\\
\none & \none & \none &  \none & \none & \none[\cdot] & \none[\cdot]& \none[\cdot]  
\end{ytableau}
\end{equation}

\item The cores may not even have the same size.
Below are two fillings where only one having an empty $3$-core.
\begin{equation}\label{eq:core3}
 \begin{ytableau}
 \cdot & *(pBlue)\cdot &  *(pPink) \cdot \\
\none &  &   *(pPink)  & *(pPink)  \\
\none &\none &  & *(pBlue) & *(pBlue) \\
\none & \none & \none & \none[\cdot] & \none[\cdot]& \none[\cdot] 
\end{ytableau}
\qquad 
 \begin{ytableau}
 *(pPink)  \cdot & *(pPink) \cdot &  *(pBlue) \cdot \\
\none & *(pPink)   &   *(pBlue)  & *(pBlue)  \\
\none &\none &  *(pSkin)  & *(pSkin) & *(pSkin) \\
\none & \none & \none & \none[\cdot] & \none[\cdot]& \none[\cdot] 
\end{ytableau}
\end{equation}
\end{itemize}

Even if a shape can be fully tiled by $k$-ribbons as seen in \eqref{eq:core3} above, 
not every partial tiling can be completed to a full tiling.

Another difficulty stems from the fact that there is no natural analog for a $k$-quotient. 
If $k$ does not divide the number of rows times $2$, we do not have a way to partition the 
border modulo $k$ that is compatible with the cylindrical boundary.
Even when $k$ does divide the number of rows
and each ribbon has a modulo $k$ value, the $k$-quotient idea does not 
translate to the cylindrical case. For the example below, there are exactly three ways to place the two ribbons with different sets of diagonal values modulo $3$.
\begin{equation}\label{eq:3fills}
 \begin{ytableau}
 *(pBlue) \cdot & *(pBlue) \cdot \\
\none & *(pBlue)& *(pPink) \\
\none & \none & *(pPink) \; & *(pPink) \\
 \none & \none & \none & \none[\cdot] & \none[\cdot]
\end{ytableau}\
\qquad 
\begin{ytableau}
 *(pBlue) \cdot & *(pPink) \cdot \\
\none & *(pPink)& *(pPink) \\
\none & \none & *(pBlue) \; & *(pBlue) \\
\none & \none & \none & \none[\cdot] & \none[\cdot] 
\end{ytableau}
\qquad 
\begin{ytableau}
*(pPink) \cdot & *(pPink) \cdot \\
\none & *(pBlue)& *(pBlue) \\
\none & \none & *(pBlue) \; & *(pPink) \\
\none & \none & \none & \none[\cdot] & \none[\cdot] 
\end{ytableau}
\end{equation}

These fundamental differences mean that we need new techniques to investigate how the ribbon tilings behave in the cylindric setting.

\section{Quasisymmetric functions}

Our proof method of the cylindric Murnaghan--Nakayama
rule is based on the power-sum expansion of
linear combinations of so called \emph{$P$-partitions} introduced by R.~Stanley~\cite{StanleyEC2}.
Thus, we need to review several concepts before getting to the actual proof.

A \defin{quasisymmetric function} $F$ is a formal power series $F \in \setQ[[x_1,x_2,\dotsc]]$
such that the degree of $F$ is finite, and for every composition $(\alpha_1,\dotsc,\alpha_\ell)$
the coefficient of $x_{i_2}^{\alpha_2} \dotsm x_{i_\ell}^{\alpha_\ell}$ in $F$
is the same for all integer sequences $1\leq i_1 < i_2 < \dotsb < i_\ell$.
Given a composition $\alpha$ with $\ell$ parts, 
the \defin{monomial quasisymmetric function} $\qmonom_\alpha$
is defined as
\[
\qmonom_\alpha(\xvec) \coloneqq \sum_{i_1 < i_2 < \dotsb < i_\ell} 
x_{i_1}^{\alpha_1} x_{i_2}^{\alpha_2} \dotsm x_{i_\ell}^{\alpha_\ell}
\,.
\]
The set $\{\qmonom_\alpha\}_\alpha$ constitutes a basis for the space of 
homogeneous quasisymmetric functions of degree $n$ as $\alpha$ ranges over all compositions of $n$.

We shall also need a \emph{quasisymmetric power sum basis} for the space of quasisymmetric functions, 
introduced in \cite{BallantineDaughertyHicksMason2020}.
In order to define these, we need some additional notation. 
For a composition $\alpha$, let $\defin{z_\alpha} \coloneqq \prod_{i\geq1} i^{m_i} m_i!$,
where $m_i$ denotes the number of parts of $\alpha$ that are equal to $i$.
Given compositions $\alpha \leq \beta$ (denoting refinement)
so that $\beta = \beta_1 + \beta_2 + \dotsb + \beta_k$ and $\beta_i = \alpha_{i,1} + \alpha_{i,2} + \dotsb + \alpha_{i,\ell_i}$
for $1 \leq i \leq k$, let
\begin{equation}
\defin{\pi(\alpha,\beta)}
\coloneqq
\prod_{i=1}^{k} 
\alpha_{i,1} \cdot \left(\alpha_{i,1}+\alpha_{i,2}\right)\dotsm 
\left(\alpha_{i,1}+\alpha_{i,2} + \dotsb + \alpha_{i,\ell_i}  \right).
\end{equation}
The \defin{quasisymmetric power sum} $\Psi_\alpha$ is then defined as
\begin{equation}\label{eq:PsiintoMonomial}
\Psi_\alpha(\xvec) \coloneqq z_\alpha \sum_{\beta \geq \alpha } \frac{1}{\pi(\alpha,\beta)} \qmonom_\beta(\xvec).
\end{equation}
The quasisymmetric power sums refine the power sum symmetric functions; 
\begin{equation}\label{eq:powersum_Psi_expansion}
\powersum_\lambda(\xvec) = \sum_{\alpha \sim \lambda} \Psi_\alpha(\xvec)
\,,
\end{equation}
where the sum ranges over all compositions $\alpha$ whose parts rearrange to $\lambda$,
see \cite[Prop.~3.19]{BallantineDaughertyHicksMason2020}.

\subsection{Quasisymmetric functions from posets}

We shall now introduce a few types of quasisymmetric 
functions arising from partially ordered sets. 

\begin{definition}
Let $P$ be a poset and let $\rel(P)$ be the set of all pairs of related elements;
$\defin{\rel(P)} \coloneqq \{ (\poseta,\posetb) \in P : \poseta <_P \posetb \}$.
Given a poset $P$, the \defin{set of weakly order-preserving maps} from $P$ to the set of positive integers
is defined as 
\[
\defin{\orderPresF(P)} \coloneqq 
\{ f: P \to \setP : \poseta <_P \posetb \implies f(\poseta) \leq f(\posetb)\}.
\]
The \defin{P-partition generating function}, introduced by R.~Stanley~\cite{Stanley1972},
is defined as
\begin{equation}
 \defin{K_P(\xvec)} \coloneqq 
\sum_{\substack{f \in \orderPresF(P)}} \prod_{\poseta \in P} x_{f(\poseta)}.
\end{equation}
It is easy to see that these are indeed quasisymmetric functions---for more information
see \cite[Sec. 7.19]{StanleyEC2}.
\medskip 

Now, let $S \subseteq \rel(P)$ be any subset of related pairs in $P$ and 
define the following quasisymmetric functions:
\begin{equation*}
 \defin{\equalK{P,S}(\xvec)} \coloneqq \hspace{-7mm}
 \sum_{\substack{f \in \orderPresF(P) \\ (\poseta,\posetb)\in S \Rightarrow f(\poseta) = f(\posetb) }} 
 \prod_{\poseta \in P} x_{f(\poseta)}, 
 \quad 
 \defin{\strictK{P,S}(\xvec)} \coloneqq \hspace{-7mm}
 \sum_{\substack{f \in \orderPresF(P) \\ (\poseta,\posetb)\in S \Rightarrow f(\poseta) < f(\posetb) }} 
 \prod_{\poseta \in P} x_{f(\poseta)}.
\end{equation*}
The generating function $\strictK{P,S}(\xvec)$ was studied earlier in \cite{McNamara2006}, where the pair $(P,S)$
is referred to as an \emph{oriented poset}.
\end{definition}

We shall need the following simple observation in the upcoming proof.
\begin{lemma}
Let $P$ be a poset and $E \subseteq \rel(P)$. Then 
\begin{equation}\label{eq:inclExcl}
 \strictK{P,E}(\xvec) = \sum_{S \subseteq E} (-1)^{|S|} \equalK{P,S}(\xvec).
\end{equation}
\end{lemma}
\begin{proof}
 This is simply inclusion-exclusion.
\end{proof}

A cylindric diagram $D$ defines a poset \defin{$P = P(D)$}, where each box is an element in $P$,
and covering relations are $\poseta \lessdot \posetb$ if $\poseta$ is a box to the left of $\posetb$,
or $\poseta$ is above $\posetb$ (in the cylindrical sense).

\begin{example}
In the following cylindric diagram, there are 13 boxes 
where the leftmost column wraps around as indicated by it being repeated.
\begin{figure}[!ht]
\[
 \ytableausetup{boxsize=1.2em}
\begin{ytableau}
\scriptstyle{13}  \\
\scriptstyle{12} \\ 
 8 & 9 & \scriptstyle{10} & \scriptstyle{11} \\
\none & \none &   6 &   7  \\
\none & \none &   3 &   4 &   5 &  *(pSkin) \scriptstyle{13}  \\
\none & \none & \none &   1 &   2 & *(pSkin) \scriptstyle{12}  \\
 \none & \none & \none &   \none &   \none & *(pSkin) 8
\end{ytableau}
\qquad 
P=
\begin{tikzpicture}[scale=1.6]
\draw
( 4,1)node[circled](v1){1}
( 6,1)node[circled](v3){3}
(10,1)node[circled](v8){8}

( 3, 2)node[circled](v2){2}
( 5, 2)node[circled](v4){4}
( 7, 2)node[circled](v6){6}
 (9,2)node[circled](v9){9}
 (11,2)node[circled](v12){$\scriptstyle{12}$}
 
 (4,3)node[circled](v5){5}
 ( 6,3)node[circled](v7){7}
 (8,3)node[circled](v10){$\scriptstyle{10}$}
 (12,3)node[circled](v13){$\scriptstyle{13}$}
 
(7,4)node[circled](v11){$\scriptstyle{11}$}
;
\draw (v1)--(v2);
\draw (v3)--(v4)--(v5);
\draw (v6)--(v7);
\draw (v8)--(v9)--(v10)--(v11);

 \draw (v13) .. controls (5,0) .. (v5);
 \draw (v12) .. controls (5,-1) .. (v2);

 \draw[thickLine] (v2)--(v5);
 \draw[thickLine] (v1)--(v4)--(v7)--(v11);
 \draw[thickLine] (v3)--(v6)--(v10);
 \draw[thickLine] (v8)--(v12)--(v13);
\end{tikzpicture}
\]
\caption{A cylindric diagram $D$ and its associated poset.
The reader is encouraged to draw $P(D)$ on a cylinder, where 
the relations $2 <_P 12$ and $5<_P 13$ ``wrap around''.
The strict edges are drawn as thick lines $D$.
}\label{fig:cylindricPoset}
\end{figure}
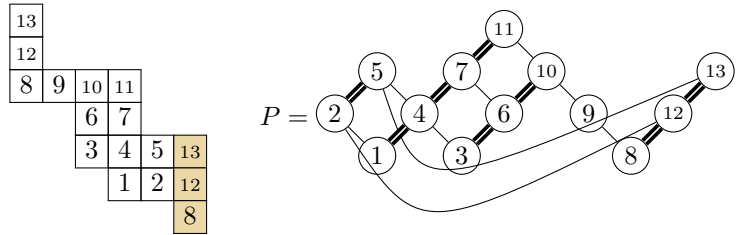
\end{example}

\begin{definition}\label{def:cylindricSkewSchur}
A \defin{cylindric semistandard Young tableau} (CSSYT) of shape $D$
is a filling of the boxes in the diagram $D$ with positive integers,
such that entries weakly increase as column index increases
and strictly increases as row index increases.
We let $\CSSYT(D)$ denote the set of all CSSYTs of shape $D$.

The \defin{cylindric skew Schur function} $\schurS_D(\xvec)$ associated 
with the diagram $D$ is then defined as 
\[
 \schurS_D(\xvec) = \sum_{T \in \CSSYT(D)} \xvec_{T}
\]
where $\xvec_T \coloneqq \prod_{(i,j) \in D} \xvec_{T(i,j)}$.
\end{definition}

\begin{example}
The following is a cylindric skew diagram with $7$ boxes,
followed by two cylindric skew semistandard Young tableau.
In the CSSYTs, the third box in the first row is 
the same as the box in the very bottom row.
\[
\ytableausetup{boxsize=1.0em}
\begin{ytableau}
*(pSkin) \cdot \\
*(pSkin) &  *(pSkin)  \\
\none & *(pSkin) & *(pSkin) \\
\none &   *(pSkin)  & *(pSkin)  & \none[\cdot] \\
\end{ytableau}
\qquad 
\begin{ytableau}
3   \\
1 &  4  \\
\none & 3 & 3 \\
\none &  1 & 2  & \color{gray}{3} \\
\end{ytableau}
\qquad 
\begin{ytableau}
3   \\
1 &  3  \\
\none & 2 & 2 \\
\none &  1 & 1  & \color{gray}{3} \\
\end{ytableau}
\]
The two CSSYT contribute with the monomials $x_1^2 x_2 x_3^3 x_4$ and $x_1^3 x_2^2 x_3^2$ to 
the corresponding cylindric skew Schur function.
The full expansion in the monomial basis is 
\[
 2 \monomial_{322}+\monomial_{331}+8 \monomial_{2221}+4 \monomial_{3211}+16 \monomial_{22111}+8 \monomial_{31111}+32 \monomial_{211111}+64 \monomial_{1111111}.
\]
\end{example}

We note that if we let $P=P(D)$ be the poset associated with $D$, 
and $E \subseteq \rel(P)$ be the edges \defin{strict edges}
corresponding to adjacent boxes in the same column (thick edges in
\cref{fig:cylindricPoset}).
It is then just a matter of unraveling the definitions, to observe that
\begin{equation}\label{eq:posetFromDiagram}
 \schurS_D(\xvec) = \strictK{P,E}(\xvec).
\end{equation}
We can now see that $\schurS_D(\xvec)$ is a quasisymmetric function.
In fact, it is a symmetric function, see \cite[Cor.5.3]{Postnikov2005}, 
and one can prove this by using a cylindric variant of Bender--Knuth involutions,
see \cite[Thm. 7.10.2]{StanleyEC2}.

\begin{remark}\label{rem:notPwPartition}
 In Stanley's theory of $(P,w)$-partitions, \cite[Sec. 7.19]{StanleyEC2}
 there is a generalization of $K_P(\xvec)$ which relies on a labeling $w$ of $P$---
 this labeling determines a subset of edges which are to be considered strict.
 The corresponding generating functions enjoy many nice properties,
 such as a positive expansion in the fundamental quasisymmetric basis.
 Unfortunately, $\schurS_D(\xvec)$ cannot in general be realized as such a generating function
 from a $(P,w)$-partition, as noted in \cite[Example 2.7]{McNamara2006}.
 For example, \cite[Thm. 5.7]{McNamara2006} states that only
 \emph{non-cylindric} skew shape diagrams (i.e., isomorphic to a regular skew shape)
 $D$ have an $\schurS_D(\xvec)$ which is positive in the fundamental quasisymmetric basis.
\end{remark}

\subsection{Chain congruences}

We shall now recall some terminology from \cite{AlexanderssonSulzgruber2019}.
Given $E \subseteq \rel(P)$, we let $\defin{\erel{E}} \supseteq E$ 
be defined via
\[
 (\poseta,\posetb) \in \erel{E} \iff f(\poseta)=f(\posetb)
 \text{ \emph{for all} } f \in \equalOrderPresF(P,E).
\]
(This is more than just taking the transitive closure of the relations in $E$; 
see \cref{ex:quotientPoset}.)
Note that for any $E \subseteq \rel(P)$, we have that
\[
 \equalK{P,E}(\xvec)=\equalK{P,\erel{E}}(\xvec).
\]
We say that $E$ is a \defin{chain congruence} (on $P$) if $E=\erel{E}$.
Moreover, $\erel{E}$ defines an equivalence relation on elements in $P$, where 
$\poseta$ and $\posetb$ are equivalent under $\erel{E}$ if $(\poseta,\posetb) \in \erel{E}$.
Finally, any subset $E \subseteq \rel(P)$ defines a \defin{quotient poset}, 
$\defin{P/\erel{E}}$, where the elements in $P/\erel{E}$ are the equivalence classes of $\erel{E}$,
and relations inherited from the original poset.
To be precise, if $\poseta', \posetb' \in P/\erel{E}$, then 
$\poseta'$ is less than $ \posetb'$ if and only if there are $\poseta, \posetb \in P$
with $\poseta <_P \posetb$, $\poseta$ belonging to the equivalence class $\poseta'$ and 
$\posetb$ being in the equivalence class $\posetb'$, see \cref{ex:quotientPoset}.

\begin{example}\label{ex:quotientPoset}
Let $P$ be the poset below, and $E = \{(1,7), (2,3),(4,6),(6,9)\}$
has been marked with thick edges.
\begin{equation}
P=
\begin{tikzpicture}
\draw
(1,1)node[circled](v1){1}
(3,1)node[circled](v2){2}
(2,3)node[circled](v3){3}
(1,5)node[circled](v4){4}
(3,5)node[circled](v5){5}
(0,7)node[circled](v6){6}
(0,3)node[circled](v7){7}
(2,7)node[circled](v8){8}
(1,9)node[circled](v9){9}
;
\draw (v7)--(v1)--(v3)--(v4)--(v8)--(v9);
\draw (v3)--(v5);
\draw[thickLine] (v2)--(v3);
\draw[thickLine] (v4)--(v6)--(v9);
\draw[thickLine] (v1)--(v7);
\end{tikzpicture}
\qquad
\qquad 
P/\erel{E} =
\begin{tikzpicture}
\draw
(1,1)node[circled](v1){$\scriptstyle{17}$}
(3,4)node[circled](v3){$\scriptstyle{23}$}
(1,7)node[circled](v4){$\scriptstyle{4689}$}
(5,7)node[circled](v5){5};
\draw (v1)--(v3)--(v4);
\draw (v3)--(v5);
\end{tikzpicture}
\end{equation}
Note that for any $f \in \equalOrderPresF(P,E)$ we have that $f(4)=f(6)=f(9)$,
and since $f(4) \leq f(8) \leq f(9)$, we must also always have $f(4)=f(8)=f(9)$.
Hence, $\erel{E}$ must contain $(4,8)$ and $(8,9)$.
In total, $\erel{E} = \{(1,7), (2,3),(4,6),(4,8),(4,9),(8,9)\}$
and $P/\erel{E}$ has four elements.
\end{example}

\subsection{Quasisymmetric powersum functions, again}

We shall need a few definitions in order to state the result we shall need.
Let $P$ be a poset on $n$ elements. Given a composition $\alpha \vDash n$
let $\defin{\opsurj_{\alpha}(P)}$ denote the set of \defin{order-preserving surjections}
$f:P\to [\length(\alpha)]$ such that $|f^{-1}(i)| = \alpha_i$ for $i=1,\dotsc,\length(\alpha)$,
and $\poseta <_P \posetb \implies f(\poseta)\leq f(\posetb)$.
Whenever $E \subseteq \rel(P)$, let
$\defin{\opsurj_{\alpha}(P,E)} \subseteq \opsurj_{\alpha}(P)$ denote the set 
of order-preserving surjections $f\in\opsurj_{\alpha}(P)$ 
which satisfy $f(\poseta)=f(\posetb)$ whenever $(\poseta,\posetb) \in E$.

Observe that if $E$ is a chain-congruence on $P$ and $f \in \opsurj_{\alpha}(P,E)$,
then each sub-poset $f^{-1}(i)$ is the union of some equivalence-classes 
of $E$. Moreover, each $f^{-1}(i)/\erel{E}$ is a sub-poset of $P/\erel{E}$.

We let $\defin{\mecs{E}{P}} \in \{0,1,2,\dotsc \}$ 
be $0$ unless $P/\erel{E}$ has a unique minimal element\footnote{Minimum for short.},
in which case we let $\mecs{E}{P}$
be the cardinality of the $\erel{E}$-equivalence class corresponding to this 
unique minimum in $P/\erel{E}$.
For example, $\mecs{E}{P}=2$ (the equivalence class $\{1,7\}$
is the unique minimum in $P/\erel{E}$)
in \cref{ex:quotientPoset}.

\begin{example}\label{ex:quotientPoset3}
Consider the following poset $P$, some relations $E \subset \rel(P)$, 
and $P/\erel{E}$ illustrated below:
\begin{equation*}
(P,E) = 
\begin{tikzpicture}
\draw
( 6,0)node[circled](v1){1}
( 8,2)node[circled](v2){3}
(10,0)node[circled](v3){2}
( 4,2)node[circled](v4){4}
( 6,4)node[circled](v5){5}
(10,4)node[circled](v6){6}
( 2,4)node[circled](v7){7}
( 4,6)node[circled](v8){8}
( 6,8)node[circled](v9){9}
( 0,6)node[circled](v0){0}
;
\draw (v1)--(v4);
\draw (v2)--(v5)--(v8)--(v7)--(v0);
\draw (v2)--(v6)--(v9);
\draw[thickLine] (v1)--(v2)--(v3);
\draw[thickLine] (v4)--(v7);
\draw[thickLine] (v8)--(v9);
\end{tikzpicture}
\qquad
P/\erel{E} = 
\begin{tikzpicture}
\draw
( 1,0)node[circled](v123){123}
(-2,3)node[circled](v47){47}
( 0,3)node[circled](v5){5}
( 2,3)node[circled](v6){6}
(-1,6)node[circled](v89){89}
(-4,6)node[circled](v0){0}
;
\draw (v123)--(v47)--(v0);
\draw (v123)--(v5)--(v89)--(v47);
\draw (v123)--(v6)--(v89);
\end{tikzpicture}
\end{equation*}
If $\alpha=(5,3,2)$, then $\opsurj_{\alpha}(P,E) = \{f_1, f_2\}$ has two elements
illustrated below:
\begin{equation*}
f_1 = 
\begin{tikzpicture}
\draw
( 6,0)node[circled,fill=pBlue](v1){1}
( 8,2)node[circled,fill=pBlue](v2){1}
(10,0)node[circled,fill=pBlue](v3){1}
( 4,2)node[circled,fill=pBlue](v4){1}
( 6,4)node[circled,fill=pSkin](v5){2}
(10,4)node[circled,fill=pSkin](v6){2}
( 2,4)node[circled,fill=pBlue](v7){1}
( 4,6)node[circled,fill=pAlgae](v8){3}
( 6,8)node[circled,fill=pAlgae](v9){3}
( 0,6)node[circled,fill=pSkin](v0){2}
;
\draw (v1)--(v4);
\draw (v2)--(v5)--(v8)--(v7)--(v0);
\draw (v2)--(v6)--(v9);
\draw[thickLine] (v1)--(v2)--(v3);
\draw[thickLine] (v4)--(v7);
\draw[thickLine] (v8)--(v9);
\end{tikzpicture}
\qquad 
f_2 = 
\begin{tikzpicture}
\draw
( 6,0)node[circled,fill=pBlue](v1){1}
( 8,2)node[circled,fill=pBlue](v2){1}
(10,0)node[circled,fill=pBlue](v3){1}
( 4,2)node[circled,fill=pSkin](v4){2}
( 6,4)node[circled,fill=pBlue](v5){1}
(10,4)node[circled,fill=pBlue](v6){1}
( 2,4)node[circled,fill=pSkin](v7){2}
( 4,6)node[circled,fill=pAlgae](v8){3}
( 6,8)node[circled,fill=pAlgae](v9){3}
( 0,6)node[circled,fill=pSkin](v0){2}
;
\draw (v1)--(v4);
\draw (v2)--(v5)--(v8)--(v7)--(v0);
\draw (v2)--(v6)--(v9);
\draw[thickLine] (v1)--(v2)--(v3);
\draw[thickLine] (v4)--(v7);
\draw[thickLine] (v8)--(v9);
\end{tikzpicture}
\end{equation*}
We see that $Q \coloneqq f^{-1}_1(1)$ is a subposet of $P$, induced by the elements $\{1,2,3,4,7\}$.
Moreover, $Q/\erel{E}$ is a subposet of $P/\erel{E}$, consisting of the two equivalence classes 
$\{123, 47\}$. We have that $\mecs{E}{ Q } = 3$, since $123$ is a 
unique minimal element in $P/\erel{E}$ and it has cardinality $3$.
However, the poset $f^{-1}_1(2)/\erel{E} = \{5,6,0\}$ does not have a unique minimum,
so $\mecs{E}{ f^{-1}_1(2) } = 0$.
The full picture is given as follows:
\begin{equation}\label{eq:mECalc}
\begin{aligned}
 \mecs{E}{ f^{-1}_1(1) } &= 3, & \mecs{E}{ f^{-1}_1(2) } &= 0,  & \mecs{E}{ f^{-1}_1(3) } &= 2, \\
 \mecs{E}{ f^{-1}_2(1) } &= 3, & \mecs{E}{ f^{-1}_2(2) } &= 2,  & \mecs{E}{ f^{-1}_2(3) } &= 2.
\end{aligned}
\end{equation}
\end{example}

We have now introduced all the necessary ingredients to state the following 
theorem, which gives the quasisymmetric power-sum expansion of $\equalK{P,E}(\xvec)$.
\begin{theorem}[{Alexandersson--Sulzgruber, \cite[Thm. 6.2]{AlexanderssonSulzgruber2019}}]\label{thm:ASFormula}
Let $P$ be a poset and let $E$ be a chain-congruence on $P$.
Then
\begin{equation}\label{eq:ASFormula}
\equalK{P,E}(\xvec)
=\sum_{\alpha\vDash n}
\frac{\Psi_{\alpha}(\xvec)}{z_{\alpha}}
\sum_{f\in\opsurj_{\alpha}(P,E)}
\prod_{j=1}^{\ell(\alpha)}
\mecs{E}{ f^{-1}(j) },
\end{equation} 
where $f^{-1}(j)\subseteq P$ is a sub-poset.
\end{theorem}
Note that all coefficients appearing in \eqref{eq:ASFormula} are non-negative.

\begin{example}
Let $P$ and $E$ be given as \cref{ex:quotientPoset3}.
Then the coefficient of $\frac{\Psi_{532}(\xvec)}{z_{532}}$ in the $\Psi$-expansion of
$\equalK{P,E}(\xvec)$ is given as
\[
 \equalK{P,E}(\xvec) =  \dotsb \; + \; \frac{\Psi_{532}(\xvec)}{z_{532}}\left( 
 3\cdot 0 \cdot 2 + 3\cdot 2 \cdot 2
 \right) \; + \; \dotsb 
\]
according to the calculations in \eqref{eq:mECalc}.
% Full expression below. Here, psi[73] represents \Psi_{73}/z_{73}
% 3 psi[10]+6 psi[55]+6 psi[73]+6 psi[82]+3 psi[91]+12 psi[415]+12 psi[433]+
% 12 psi[532]+6 psi[541]+12 psi[613]+6 psi[631]+18 psi[712]+6 psi[721]+
% 12 psi[3115]+12 psi[3133]+12 psi[3313]+24 psi[4132]+12 psi[4141]+12 psi[4213]+
% 12 psi[4231]+24 psi[4312]+12 psi[5113]+6 psi[5131]+12 psi[5212]+
% 12 psi[5221]+36 psi[6112]+12 psi[6121]+24 psi[31132]+12 psi[31141]+
% 12 psi[31213]+12 psi[31231]+24 psi[31312]+24 psi[32113]+12 psi[32131]+
% 24 psi[33112]+24 psi[41212]+24 psi[41221]+48 psi[42112]+24 psi[42121]+
% 36 psi[51112]+12 psi[51121]+24 psi[311212]+24 psi[311221]+48 psi[312112]+
% 24 psi[312121]+72 psi[321112]+24 psi[321121]
\end{example}

\section{A Murnaghan--Nakayama rule for cylindric Schur functions}

Recall that skew Schur functions $\schurS_{\lambda/\mu}(\xvec)$ is a 
$(P,w)$-partition generating function $K_{P,w}(\xvec)$ for a special type of labeled poset.
Liu and Weselcouch~\cite{LiuWeselcouch2020} found a formula for the
$\Psi$-expansion of $K_{P,w}(\xvec)$ in general.
The Liu--Weselcouch formula is a uniform generalization of 
the classical Murnaghan--Nakayama formula for skew Schur functions,
and the $E=\emptyset$ case of \cref{thm:ASFormula}.

However, we noted in \cref{rem:notPwPartition} that a \emph{cylindrical} skew Schur function
$\schurS_D(\xvec)$ does rarely coincide with some $K_{P,w}(\xvec)$
so we must use a different strategy for cylindric diagrams $D$.

\begin{proposition}\label{prop:reduction}
Let $P$ be a poset and $E \subseteq \rel(P)$. 
Then 
\begin{align}\label{eq:reduction}
\strictK{P,E}(\xvec) &= 
\sum_{\alpha\vDash n}
\frac{\Psi_{\alpha}(\xvec)}{z_{\alpha}}
\sum_{f \in \opsurj_{\alpha}(P,E)}
\prod_{j=1}^{\ell(\alpha)}
\left(
\sum_{S_j \subseteq E\cap \rel( f^{-1}(j) )} (-1)^{|S_j|} \mecs{\erel{S_j}}{f^{-1}(j)}
\right)
\end{align}
where in the last sum, $S_j$ ranges over all subsets of $E$,
which only involves elements in the sub-poset $f^{-1}(j)$.
\end{proposition}
\begin{proof}
We combine the formula for $\equalK{P,E}(\xvec)$ in
\eqref{eq:ASFormula} with the inclusion-exclusion formula \eqref{eq:inclExcl}.
This leads to
\[
\strictK{P,E}(\xvec) = 
\sum_{S \subseteq E}
(-1)^{|S|}
\sum_{\alpha\vDash n}
\frac{\Psi_{\alpha}(\xvec)}{z_{\alpha}}
\sum_{f\in\opsurj_{\alpha}(P,S)}
\prod_{j=1}^{\ell(\alpha)}
\mecs{ \erel{S} }{ f^{-1}(j) }.
\]
We can then change the order of summation, 
and choose the order-preserving surjection $f$ before choosing the subset $S$. 
This leads to
\[
\strictK{P,E}(\xvec) = 
\sum_{\alpha\vDash n}
\frac{\Psi_{\alpha}(\xvec)}{z_{\alpha}}
\sum_{f\in\opsurj_{\alpha}(P,E)}
\sum_{S \subseteq E}
\prod_{j=1}^{\ell(\alpha)}
(-1)^{|S|}
\mecs{ \erel{S} }{ f^{-1}(j) }.
\]
Now, the subsets $S = S_1 \cup S_2 \cup \dotsb \cup S_{\length(\alpha)}$ 
can be constructed as disjoint unions
where the $S_j$ are chosen independently on the sub-posets $f^{-1}(j)$.
(Recall that in the definition of $\opsurj_{\alpha}(P,S)$,
we must have $f(\poseta)=f(\posetb)$ whenever $(a,b) \in S$,
so $S$ indeed restricts to independent 
choices of $S_j \subseteq E\cap \rel( f^{-1}(j) )$.)
This now leads to the final expression in \eqref{eq:reduction}.
\end{proof}

The remaining part of this section is dedicated to simplifying the innermost sum of \eqref{eq:reduction}.
This is equivalent to computing the coefficient
of $\Psi_{(n)}/z_{(n)}$ in $\strictK{P,E}(\xvec)$.
This inner sum is simply
\begin{equation}\label{eq:mainFocusFirstReduction}
  \sum_{S \subseteq E} (-1)^{|S|} \mecs{\erel{S}}{P}.
\end{equation}
Let us compute the coefficient
of $\Psi_{(n)}/z_{(n)}$ for $\strictK{P,E}(\xvec)$ in
\cref{ex:quotientPoset3}.
Note first that $P / \erel{E}$ has a unique
minimum $M$ (namely $123$).
However, $(4,7) \in E$ is not part of $M$
so by adding or removing the edge $(4,7)$ in $S$
we obtain a sign-reversing involution. Hence,
for this particular poset and set of strict edges (edges representing strict inequality), \eqref{eq:mainFocusFirstReduction} is $0$.

\subsection{Stacked ribbons}

A cylindric diagram $D$ on $\mathfrak{C}_{x,y}$ corresponds to a poset $P$ with strict edges $E$ as in \eqref{eq:posetFromDiagram}.

By removing loop ribbons iteratively starting from the outer boundary,
we can write $D$ as the disjoint union
\[
 D =   L_1 \cup L_2 \cup \dotsb \cup L_\ell \cup F
\]
each $L_i$ is a loop ribbon and where $F$ is a either also
a loop ribbon or a non-empty cylindric diagram not containing a loop (an $x+y$-core of the shape).
We say that $D$ is a \defin{pure stacked ribbon} if $F$ is a loop ribbon
and a \defin{(non-pure) stacked ribbon} if $F$ is a (non-loop) ribbon.
Note that the only other cases are when $F$ is disconnected or 
contains a $2{\times}2$-box.
We have previously defined the height of a ribbon as the number of rows it occupies for non-loop ribbons and $y$ for loop ribbons. In the language of posets, the height of a ribbon is the number of edges in $E$.
We extend the definition to stacked ribbons by setting the \defin{height} of a stacked ribbon
$D =   L_1 \cup L_2 \cup \dotsb \cup L_\ell \cup F$
to be the sum of the heights of its constituents (We can equivalently define the height as the number of vertical edges
of an extended ribbon wrapping around the shape multiple times and cylindrically covers the entire shape.
Such a ribbon only exists if $D$ is a stacked ribbon
and it is unique for \emph{non-pure} stacked ribbons.) As all loop ribbons have height $y$, we have $\height(D)=\ell y+\height(F)$.

The \defin{width} of a pure stacked ribbon is defined as the 
number of (cylindrical) columns it occupies.

\begin{example}
Below are five examples of stacked ribbons
and the correspond heights for these.
\begin{align*}\label{eq:stackedRibbons}
\ytableausetup{boxsize=0.8em}
&
\begin{ytableau}
*(pSand) \cdot \\
*(pSand)  & *(pSand) \\
\none & *(pSand) & *(pSand) \\
\none & \none & \none & \none[\cdot]
\end{ytableau}
&&
\begin{ytableau}
*(pSkin) \cdot \\
*(pSkin)  & *(pSkin) & *(pSkin) \\
\none & \none & *(pSkin) & *(pSkin) & \none[\cdot]
\end{ytableau}
&&
\begin{ytableau}
*(pBlue) \cdot & *(pBlue) & *(pBlue) & *(pBlue) \\
*(pSkin) \cdot  & *(pSkin) & *(pSkin) & *(pBlue) & \none[\cdot] \\
\none & \none & *(pSkin) & *(pSkin) & \none[\cdot]
\end{ytableau}
&&
\begin{ytableau}
*(pSkin) \cdot \\
*(pSkin) \\
*(pSkin) &*(pSkin) &*(pSkin) &*(pSand) &*(pSand) \\ 
\none & \none &  *(pSkin) &  *(pSkin) & *(pSand) \\
\none & \none & \none &  *(pSkin) &  *(pSand) & *(pSand) \\
\none & \none & \none &  *(pSkin) &  *(pSkin) & *(pSand) \\
\none & \none & \none & \none &  *(pSkin) & *(pSkin) & *(pSkin) & \none[\cdot] 
\end{ytableau}
&&
\begin{ytableau}
*(pBlue) \cdot & *(pBlue) & *(pSand)\\
*(pSkin) \cdot & *(pBlue) & *(pSand) \\
*(pSkin) \; & *(pBlue) & *(pBlue) & *(pBlue) & \none[\cdot] \\
*(pSkin) \; & *(pSkin) & *(pSkin) & *(pSkin) & \none[\cdot] 
\end{ytableau}
\\
&2 && 2 && 1+1 && 3+5 && 1+2+2
\end{align*}
The second and third stacked ribbon are pure
and both of these have width 4.
\end{example}

\begin{proposition}\label{prop:stackedRibbon}
Let $D$ be a cylindrical skew shape where $E$ is the set of strict edges. 
Then
\begin{equation}\label{eq:mnrule}
 \sum_{S\subseteq E}(-1)^{|S|}m_S(P) = 
 \begin{cases}
  (-1)^{\height(D)}  &\text{ if $D$ is a non-pure stacked ribbon,} \\
  \width(D) \cdot (-1)^{\height(D)}  &\text{ if $D$ is a pure stacked ribbon,} \\
  0  &\text{ otherwise.} \\
 \end{cases}
\end{equation}
Note that the first case implies the classical Murnaghan--Nakayama rule.
\end{proposition}
Proof of this proposition involves a dedicated use of inclusion-exclusion through a 
number of different cases and restrictions. It is postponed to the appendix.
\medskip 

Let $D$ be cylindric diagram with associated poset $P$  and let $\mu$ be a partition.
The set of \defin{ribbon tableaux} of shape $D$ and content $\mu$
is the set of order-preserving surjections $f \in \opsurj_{\mu}(P_D)$
such that each poset $f_j \coloneqq f^{-1}(j)$ is a stacked ribbon.
Let us denote this set by $\RT(D,\mu)$.
By \cref{prop:stackedRibbon}, it suffices to sum over only such order-preserving surjections
when computing cylindric Schur functions using the formula in \cref{prop:reduction}.
In other words, we have simplified the innermost sum in \eqref{eq:reduction}.
This simplification leads to the following theorem.
\begin{theorem}[The cylindric Murnaghan--Nakayama rule]\label{thm:cylindricMNRule}
Let $D$ be a cylindic skew Young diagram. Then 
\begin{equation}
\schurS_D(\xvec) = \sum_{\mu} \frac{\powersum_\mu(\xvec)}{z_\mu} \sum_{ f \in \RT(D,\mu) }
\prod_{j=1}^{\length(\mu)} (-1)^{\height(f_j)} \width(f_j).
\end{equation}
By convention, $\width(f_j)\coloneqq 1$ if $f_j$ is a non-pure stacked ribbon.
\end{theorem}

Now note that we can illustrate a ribbon tableau $f$ by 
tiling the diagram $D$ by stacked ribbons of the appropriate size
where the ribbon $f_j$ is labeled $j$.
This is how we usually illustrated ribbon tableaux in
classical Murnaghan--Nakayama rule. 

\begin{example}\label{ex:mnrule}
Let $D$ be the cylindric diagram below:
\[
\ytableausetup{boxsize=1.0em}
\ytableaushort{
{\none[\cdot]}{\,}{\,},
{\none}{\,}{\,},
{\none}{\none}{\cdot}
}
\]
We have the following ribbon tableaux for respective partition $\mu$.
\begin{align*}
\mu=(5):& \quad
\ytableaushort{
{\none}{1}{1},
{\none}{1}{1},
{\none}{\none}{1}
} \quad \text{(Non-pure ribbon with height 3)}
\\
\mu=(4,1):& \quad
\ytableaushort{
{\none}{1}{2},
{\none}{1}{1},
{\none}{\none}{1}
} \quad \text{(Pure with width 2 and a single box ribbon)}
\\
\mu=(3,1,1):& \quad
\ytableaushort{
{\none}{1}{3},
{\none}{1}{2},
{\none}{\none}{1}
} \quad 
\ytableaushort{
{\none }{2}{3},
{\none}{1}{1},
{\none}{\none}{1}
} \quad 
\\
\mu=(2,2,1):& \quad\emptyset \\
\mu=(2,1,1,1):&\quad \emptyset \\
\mu=(1,1,1,1,1):& \quad
\ytableaushort{
{\none }{3}{5},
{\none}{2}{4},
{\none}{\none}{1}
} \quad 
\ytableaushort{
{\none }{4}{5},
{\none}{2}{3},
{\none}{\none}{1}
} \quad 
\ytableaushort{
{\none }{3}{5},
{\none}{1}{4},
{\none}{\none}{2}
} \quad 
\ytableaushort{
{\none }{4}{5},
{\none}{1}{3},
{\none}{\none}{2}
}
\end{align*}
This gives the power-sum expansion
\[
 \schurS_D(\xvec) = -\frac{\powersum_5(\xvec)}{z_5} + 2 \frac{\powersum_{41}(\xvec)}{z_{41}}
 - 2 \frac{\powersum_{311}(\xvec)}{z_{311}}
 + 4 \frac{\powersum_{11111}(\xvec)}{z_{11111}}.
\]
\end{example}

\section{Cylindric ribbon invariants}\label{sect:invar}

In the non-cylindric case, the Murnaghan--Nakayama rule is cancellation-free
when computing $\chi^{\lambda}(\mu)$ if all parts of $\mu$ are the same.
This was proved in \cite[Cor. 9]{White1983} and \cite[Thm. 2.7.27]{JamesKerber1984}.
One way to prove this theorem is to show that all
ribbon tilings appearing in the sum for computing
$\chi^{\lambda}(\mu)$ are connected via certain local moves (flips) which preserve
the sign; see \cite{AlexanderssonPfannererRubeyUhlin2020} for a proof using this strategy.
One can ask if this still holds in the cylindrical setting,
but this turns out to be false, as we shall see in the following example.

\begin{example}
Consider the following cylindric diagram and $\mu=(2,2,2)$.
There are a total of two ribbon tilings of this diagram
using ribbons of size 2 as shown in \eqref{eq:ribbonDifferentSigns}.
The $2$-ribbon filling on the left has total height one,
whereas the filling on the right has total height two.
\begin{equation}\label{eq:ribbonDifferentSigns}
\begin{ytableau}
*(pBlue) \cdot \\
*(pSkin) \\
*(pSkin) \\
*(pOrange)& *(pOrange)&  *(pBlue) &  \none[\cdot]
\end{ytableau}
\qquad
\begin{ytableau}
*(pOrange) \cdot \\
*(pOrange) \\
*(pSkin) \\
*(pSkin)&  *(pBlue) &  *(pBlue) & \none[\cdot]
\end{ytableau}
\end{equation}

The issue is not isolated to $2$-ribbons
and the example can be adapted for any $k$ by considering
fillings of the shape of size $(2k,1,k-1,1)$.
The cases of $k=3$ and $k=4$ are shown below in \eqref{fig:kce}.
\begin{equation}\label{fig:kce}
\begin{array}{l}
\begin{ytableau}
\none[\cdot] \\
 *(pOrange)& *(pOrange) \\
\none & *(pOrange) \\
\none &  *(pSkin)& *(pSkin)&  *(pSkin)&   *(pBlue) & *(pBlue) & *(pBlue) \cdot \\
\end{ytableau}
\\
\begin{ytableau}
\none[\cdot] \\
 *(pOrange) &*(pOrange)\\
\none & *(pSkin) \\
 \none & *(pSkin)&   *(pSkin)& *(pBlue) & *(pBlue) &  *(pBlue) &*(pOrange) \cdot \\
\end{ytableau}
\end{array}
\qquad
\qquad 
\begin{array}{l}
\begin{ytableau}
\none[\cdot] \\
  *(pOrange) &*(pOrange)& *(pOrange) \\
 \none &\none & *(pOrange) \\
 \none &\none &  *(pSkin)& *(pSkin)& *(pSkin)& *(pSkin)& *(pBlue) &  *(pBlue) & *(pBlue) & *(pBlue) \cdot \\
\end{ytableau}
\\
\begin{ytableau}
\none[\cdot] \\
 *(pOrange) & *(pOrange) &*(pOrange)\\
 \none &\none & *(pSkin) \\
 \none & \none & *(pSkin)&   *(pSkin)&*(pSkin)& *(pBlue) &*(pBlue) & *(pBlue) &  *(pBlue) &*(pOrange) \cdot \\
\end{ytableau}
\end{array}
\end{equation}
The examples of $3$-ribbon (left) and $4$-ribbon (right) fillings with total heights of different parity.
\end{example}

We shall now discuss under what circumstances we \emph{can} expect the 
cylindrical Murnaghan--Nakayama rule is cancellation-free.
As noted before, the notion of \emph{$k$-quotient} does not translate to the cylindrical setting,
so we need a different framework.
\bigskip

Consider a cylindric diagram $D$ on $\mathfrak{C}_x,y$. Consider the binary word $\boun(D)$ corresponding to the outer boundary of $D$.
$\boun(D)$ is a word of length $y+x$, containing $y$ $\bw{1}$'s and $x$ $0$'s. If $k|y+x$, we can consider divide it into $k$ parts as follows: 
The \defin{$k$-partitioning} of $\boun(D)$ is a list $(\boun_1(D),\boun_2(D),\ldots,\boun_k(D))$ 
of disjoint subwords where each $\boun_i(D)$ is given by the letters whose indices are equivalent to $i$ modulo $k$.
We also denote the number of $\bw{1}$'s in one iteration of $\boun_i(D)$ by $y_i$,  $y=\sum_i y_i$. 
For example if $\boun(D)=\bw{110 001 011 010}$, it has the $3$-partitioning $(\bw{1000},\bw{1011},\bw{0110})$. 
We get $y_1=1$, $y_2=3$ and $y_3=2$.

\begin{definition}
The pair $(D,k)$ is a \defin{good pair} if for each $i$, $y_i$ satisfies one of the following:
\begin{itemize}
\item $y_i$ is equivalent to $y$ modulo $2$,
\item $y_i=0$,
\item $y_i=(x+y)/k$ (equivalently, $\boun_i$ has no $0$s) and this is an odd number. 
\end{itemize}
\end{definition}

We call diagrams whose inner boundaries are staircase shape (corresponding to $\bw{101010\ldots}$) \defin{inner-strict}. 
\begin{proposition}\label{prop:leftstrict}
If an inner-strict cylindric diagram $D$ with $s$ rows has a valid $k$-ribbon 
tiling and $k\mid s$, then $(D,k)$ are a good pair.
\end{proposition}
\begin{proof}
%\Per{This is some type of invariance proof? Is there also induction here?}
Since one can not add any even-length ribbons to a diagram with boundary $\bw{010101\ldots}$, $k$ must be odd. 
As adding or removing ribbons does not change the number of $\bw{1}$s in any $\boun_i(D)$, 
each $\boun_{i}$ has $s/k$ $\bw{1}$s, just like the inner boundary.
\end{proof}

\begin{definition} 
Given an integer $k \geq 1$ and a binary word $\wvec$ of length $nk$,
we let the \defin{pairing polynomial} be the non-commutative polynomial
\[
   \pair_k(\wvec) \coloneqq \sum_{\substack{1\leq i < j<k}} p_{ij}x_{[i]} x_{[j]}-p_{ji}x_{[j]} x_{[i]}
\]
where $p_{ij}$ denotes the number 
of pairs $(a,a+t)$ such that $a\equiv i (mod k)$,  $a+t\equiv j (mod k)$, $w_a=w_{a+t}=\bw{1}$.
The value 
\[
   \invar_k(\wvec) \coloneqq  \lfloor  \sum_{\substack{1\leq i < j<k, }}( p_{ij} -p_{ji}) /2\rfloor.
\]  
is called the \defin{$k$-pairing constant} of $\wvec$. 
We denote the parity of $\wvec$ by $\epsilon_k(\wvec)\in\{0,1\}$.
\end{definition}

\begin{lemma}\label{lem:rotationind}
For a cylindric diagram $D$ with boundary word $\wvec$, 
$\epsilon_k(\wvec)$ is rotation-independent if and only if $(D,k)$ is a good pair.
\end{lemma}
\begin{proof} 
Let $\wvec=\bit \hat{\wvec}$ where $\bit \in \{\bw{0},\bw{1}\}$. 
Let $p_{ij}$ denote the number of pairs $(a,a+t)$ such that $a\equiv i (\text{mod} k)$,  $a+t\equiv j (\text{mod} k)$, $w_a=w_{a+t}=\bw{1}$ 
for $\wvec$ and $q_{ij}$ denote the same thing for the rotated word $\hat{\wvec}\bit$. 

As the indices shift under rotation, for $i,j>1$, we have $q_{i-1j-1}=p_{ij}$. 
The case $i$ or $j$ is equal to $1$ depends on whether $\bit$ is $\bw{1}$ or $\bw{0}$.

If $\bit=\bw{0}$ we have $p_{1j}=q_{kj-1}$ and $p_{i1}=q_{i-1k}$. 
We get $\pair_k(\hat{\wvec}\bit) = \pair_k(\wvec) - \sum_j p_{1j}- p_{j1}$. 
In particular every modulo $1$ and different modulo $i$ pair affects $\epsilon_k(\wvec)$ by $1$, 
and parity remains unchanged if and only if $y_1 (y-y_1)$ is even. 

If $\bit=1$, we also have one $\bw{1}$ that flips with all the $\bw{1}$s whose indices are 
not $1$ modulo $k$. We have a total of $(y-y_1)$ of those. 
So the parity remains unchanged if and only if $(y_1-1) (y-y_1)$ is even.

Parity remains unchanged under rotation if for each $i$ 
either $y-y_i$ is even, $y_i=\bw{0}$ (so that $\bit=\bw{1}$ case never happens) or $y_i=
\ell(\boun(D))$ is odd (so that $\bit=\bw{1}$ case works and $\bit=\bw{0}$ case never happens).
\end{proof}

\begin{theorem}\label{thm:parity}
If a cylindric diagram $D$ can be tiled with $k$-ribbons with an empty core, 
where $(D,k)$ are a good pair, then for every such ribbon tiling, 
the total height of the ribbons has the same parity.
In particular, the sum in Theorem~\ref{thm:cylindricMNRule} is cancellation-free.
\end{theorem}
\begin{proof} 
If $k=\ell(\boun(D))=x+y+1$, then any $k$-ribbon is a loop and has height $x+1$, so the result trivially holds. 
Assume $k\leq x+y$. Pick a removable ribbon $R$, and let $D'=D/R$, where $D'$ has outer boundary word $\wvec'$. 
We will show that $\epsilon_k(\wvec)=\epsilon_k(\wvec')$ if and only if the height of $R$ is even. 
Removing $R$ corresponds to exchanging a $1$ at a position $t+k$ with a $0$ at position $t$ modulo $|\boun(D)|$. 
By Lemma~\cref{lem:rotationind} above, we can rotate the boundary word without 
changing $\epsilon_k(\wvec)$, so that we can assume $t<t+k\leq |\boun(D)|$. 
The height of $R$ is equal to the number of $1$'s whose indices lie in $(t,t+k)$ in the 
boundary word. The change in $\invar_k(\wvec)$ (up to rotation) is also measured by the 
number of the same number, as the $1$'s with indices in $(t,t+k)$ are flipped with the $1$ 
who moves from index $t+k$ to $t$. 
Each such pair of $1$s causes $p_{ij}$ to decrease by one and $p_{ji}$ to increase 
by one for some $i,j$, so the constant $\invar_k(\wvec)$ is altered by one (increased or decreased). 
As a consequence,  $\epsilon_k(\wvec)$ is altered if and only if the height of $R$ 
is odd, and it stays the same otherwise.

As removing a ribbon does not change the number of $1$'s in any modulo, $(D',k)$ 
is also a good pair. We can repeat the process removing 
ribbons until we reach the inner boundary $\wvec_{inner}$ of the shape. 
The constant $\epsilon_k(\wvec_{inner})$ is equal to  $\epsilon_k(\wvec)$ if and only 
if the total height of the removed ribbons is an even number. 
As this value only depends on the inner boundary and not the choice of the ribbons, 
the total height of ribbons always has the same parity.
\end{proof}

\begin{corollary}\label{cor:cancellationFree}
If an inner-strict cylindric diagram $D$ can be tiled with $k$-ribbons, 
then the parity of the total height of the ribbons is independent of the tiling.
\end{corollary}
\begin{proof} 
We can again focus on the case $k\leq x+y$, as if $k=x+y+1$ all ribbons are 
loop ribbons and they all have the same height. Note that if $D$ is non-empty, $k$ must be odd, 
as the inner boundary has $1$ on every other entry, and there is no way to 
exchange a $1$ and a $0$ whose indices differ by an even number. 
If $k$ divides the number of rows of $D$, Proposition~\ref{prop:left strict} says $(D,k)$ 
is a good pair, so result follows directly from Theorem~\ref{thm:parity} above. 
Otherwise, consider the cylindric shape $D^k$ formed by taking $k$ copies of $D$. 
% \Per{What does this copying mean? Stretching?}
This is also an inner-strict cylindric diagram. 
Furthermore, any filling of $D$ with $k$-ribbons 
can be seen as a filling of $D^k$ where every ribbon is repeated $k$ times. 
Because $k$ is odd, this does not change the parity of the total height of the ribbons. 
As $(D^k,k)$ is a good pair the result follows. 
\end{proof}
% 
% \Per{There is a recent preprint, Propp et al, who does something with cylindric tilings. \cite{DefantLiProppYoung2023} }
% \Per{Application - show that triangular grid triangle cannot be tiled with I-triominos.}

\section{Row-flagged skew Schur functions}

The flagged Schur functions were first introduced in \cite{LascouxSchutzenberger1982},
in relation to Schubert calculus.
The \defin{row-flagged skew Schur functions} were further studied by M.~Wachs~\cite{Wachs1985},
where she proved various identities involving divided difference operators.
They are defined as a sum over semi-standard Young tableaux but with a \defin{flag condition},
so that the entries in each row are restricted.
Let $\lambda/\mu$ be a skew shape with $\ell$ rows and $\avec = (a_1,\dotsc,a_\ell)$ 
and $\bvec = (b_1,\dotsc,b_\ell)$ are weakly increasing lists of non-negative integers.
Then let $\defin{\SSYT(\lambda/\mu,\avec,\bvec)}$
be the set of semi-standard Young tableaux of shape $\lambda/\mu$,
where the entries in row $j$ of $T$ are from the set $\{a_i, a_i+1,\dotsc,b_i\}$.

The \defin{row-flagged skew Schur function} is defined as
\[
 \defin{\schurS_{\lambda/\mu}(\avec,\bvec)} \coloneqq
 \sum_{T \in \SSYT(\lambda/\mu,\avec,\bvec)} \xvec_{T}
\]
Observe that this is a polynomial and not necessarily symmetric.

Similarly, the \defin{row-flagged complete homogeneous polynomial}
$\completeH_m(a,b)$ is defined as a row-flagged Schur polynomial with one row:
\[
\defin{\completeH_m(a,b)} \coloneqq
\begin{cases}
\schurS_{(m)}(a,b) & \text{if $m\geq 1$,}\\
1& \text{if $m = 0$,}\\
0& \text{otherwise.}
\end{cases}
\]
Note that this is simply the complete homogeneous symmetric polynomial of degree $m$
in the variables $\{x_a,x_{a+1},\dotsc,x_b \}$.
It will be convenient to define
\begin{equation}\label{eq:hprod}
 \defin{\completeH_\mu(\avec,\bvec)} \coloneqq \completeH_{\mu_1}(a_1,b_1)\completeH_{\mu_2}(a_2,b_2) \dotsm
 \completeH_{\mu_\ell}(a_\ell,b_\ell).
\end{equation}
Note that this is \emph{not} a symmetric function, but simply a polynomial.

I.~Gessel showed that the $\schurS_{\lambda/\mu}(\avec,\bvec)$ satisfy a Jacobi--Trudi type identity.
A different proof can be found in~\cite{Wachs1985}.
The identity states
\begin{equation}\label{eq:JT}
 \schurS_{\lambda/\mu}(\avec,\bvec) =
 \left| h_{\lambda_i-\mu_j - i + j}(a_j,b_i) \right|_{1 \leq i,j \leq \ell}.
\end{equation}

The \defin{flagged skew Kostka coefficient} $\defin{K_{\lambda/\mu,\nu}(\avec,\bvec)}$ is defined
as the coefficient of $\monomial_\nu$ in $\schurS_{\lambda/\mu}(\avec,\bvec)$.

\subsection{Gelfand--Tsetlin patterns}

Gelfand--Tsetlin patterns \defin{(GT-patterns)} were introduced
by I. Gelfand and M. Tsetlin in \cite{GelfandTsetlin1950}.
A short introduction to GT-patterns can be found in \cite[p.103]{StanleyEC2}.
There is a natural map between a GT-pattern and a semi-standard Young tableau.
A \defin{GT-pattern} is a parallelogram arrangement of integers, see \eqref{eq:gtpatternDef},
where entries must satisfy the conditions
\begin{equation*}
x_{i+1,j} \geq x_{ij} \text{ and } x_{ij} \geq x_{i+1,j+1} \label{eq:gtinequalities}
\end{equation*}
for all values of $i$, $j$ where the indexing is defined.
\begin{equation}\label{eq:gtpatternDef}
\setcounter{MaxMatrixCols}{20}
\begin{matrix}
x_{n1} & & x_{n2} & & \cdots & & \cdots & & x_{nm} \\
 & \ddots & & \ddots &  & &   & & & \ddots  \\
%    &  & x_{13} &  & x_{23}  &  & x_{33}   &  & \\
%    &  &  &  x_{12} &   & x_{22}  &     &  & \\
   &  &   x_{11} &  & x_{12} & & \cdots & & \cdots & & x_{1m} \\
   &  &    & x_{01} &    & x_{02} & & \cdots & & \cdots & & x_{0m}
\end{matrix}
\end{equation}
The inequalities simply state that horizontal rows are weakly decreasing,
down-right diagonals are weakly decreasing and down-left diagonals are weakly increasing.

The following result is well-known (and easy to prove).
\begin{lemma}[Gelfand--Tsetlin patterns and flagged SSYTs]\label{lem:flaggedGT}
Let $\lambda/\mu$ be a skew shape with $m$ rows and
let $\avec,\bvec$ be flags with $n$ entries.
The flagged semi-standard Young tableaux in
the set $\SSYT(\lambda/\mu,\avec,\bvec)$ are then
in bijection with Gelfand-Tsetlin
patterns $G$ satisfying the following conditions:
\begin{itemize}
 \item The top row of $G$ is given by $\lambda$; $x_{nj}=\lambda_j$ for all $j$.
 \item The bottom row of $G$ is given by $\mu$; $x_{0j}=\mu_j$ for all $j$.
 \item $x_{i,j}=x_{i-1,j}$ whenever $i \notin \{a_j,a_j+1,\dotsc,b_j\}$.
\end{itemize}
The bijection between patterns and tableaux is as follows:
The difference $x_{i,j}-x_{i-1,j}$ ($i\geq 1$) is the number of entries in row $j$
in the tableau equal to $i$.
\end{lemma}
Observe that the extra flag conditions corresponds to restricting to
a certain face of the (ordinary) Gelfand--Tsetlin polytope.

It is also possible to model cylindric skew
semi-standard Young tableaux using Gelfand--Tsetlin patterns.
Now it is convenient to describe cylindric diagrams
by wrapping the rows rather than columns.
The cylindric condition translates to
the GT-pattern repeating left (right) horizontally,
with a shift up (down) of all the entries.
This shift of the entries is the number
of horizontal steps between repeating columns in
the cylindric diagram; see the examples below.

\begin{example}
Let $T$ be the SSYT below (with maximal entry $4$).
Note that we shall describe cylindric diagrams by cylindrically
wrapping rows rather than columns as this is more convenient
when considering Gelfand--Tsetlin patterns.
\[
\ytableausetup{boxsize=1.0em}
\ytableaushort{
{\cdot}{\cdot},
{}{}{}{},
{\none}{}{}{},
{\none}{\none}{\none}{}{}{}{},
{\none}{\none}{\none}{\none}{\none}{\none[\cdot]}{\none[\cdot]}
}
\qquad
\begin{ytableau}
2 & 3\\
1 & 2 & 3 & 4 \\
\none & 1 & 2 & 2 \\
\none & \none &\none & 1 & 2 & 4 & 4 \\
% \none & \none & \none & \none &\none & *(pSand) 2 & *(pSand) 3 \\
\end{ytableau}
\]
This CSSYT corresponds to the cylindric GT-pattern below:
\[
\begin{matrix}
\color{gray}{\cdots}&\color{gray}{7} && 7 && 4 && 4 && 2 && \color{gray}{1} & \color{gray}{\cdots} \\
&\color{gray}{\cdots}&\color{gray}{7} && 5 && 4 && 3 && 2 && \color{gray}{0} & \color{gray}{\cdots} \\
&&\color{gray}{\cdots}&\color{gray}{6} && 5 && 4 && 2 && 1  && \color{gray}{0}& \color{gray}{\cdots} \\
&&&\color{gray}{\cdots}&\color{gray}{5} && 4 && 2 && 1 && 0  && \color{gray}{-1}& \color{gray}{\cdots} \\
&&&&\color{gray}{\cdots}&\color{gray}{5} && 3 && 1 && 0 && 0  && \color{gray}{-2} & \color{gray}{\cdots} \\
\end{matrix}
\]
The central $5 \times 4$-parallelogram is repeated horizontally, but we add (subtract) a multiple of $5$ to every entry as we move left (right).
This corresponds to the fact that in the CSSYT,
the tableau repeats every $5$ columns.
\end{example}

\begin{example}\label{ex:cylindricGT}
Let $D$ be the cylindric shape shown below.
There are three cylindric SSYTs of this shape with weight $\beta = 2211$.
\[
\ytableausetup{boxsize=1.0em}
\ytableaushort{
{\cdot}{\cdot}{\cdot},
{\none}{}{},
{\none}{\none}{}{},
{\none}{\none}{\none[\cdot]}{\none[\cdot]}{\none[\cdot]}
}
\qquad
\qquad
\ytableaushort{
{1}{2}{4},
{\none}{1}{3},
{\none}{\none}{2}{5},
{\none}
}
\qquad
\ytableaushort{
{1}{2}{5},
{\none}{1}{3},
{\none}{\none}{2}{4},
{\none}
}
\qquad
\ytableaushort{
{1}{2}{5},
{\none}{1}{4},
{\none}{\none}{2}{3},
{\none}
}
\]
The corresponding cylindric GT-patterns are
\[
\setlength{\arraycolsep}{2.5pt}
\renewcommand{\arraystretch}{0.8}
\begin{matrix}
4 &   & 3 &   & 3\\
 & 3 &   & 3 &   & 3\\
 &  & 3 &   & 3 &   & 2\\
 &  &  & 3 &   & 2 &   & 2\\
 &  &  &  & 2 &   & 2 &   & 1\\
 &  &  &  &  & 2 &   & 1 &   & 0\\
\end{matrix}
\begin{matrix}
4 &   & 3 &   & 3\\
 & 4 &   & 3 &   & 2\\
 &  & 3 &   & 3 &   & 2\\
 &  &  & 3 &   & 2 &   & 2\\
 &  &  &  & 2 &   & 2 &   & 1\\
 &  &  &  &  & 2 &   & 1 &   & 0\\
\end{matrix}
\begin{matrix}
4 &   & 3 &   & 3\\
 & 4 &   & 3 &   & 2\\
 &  & 4 &   & 2 &   & 2\\
 &  &  & 3 &   & 2 &   & 2\\
 &  &  &  & 2 &   & 2 &   & 1\\
 &  &  &  &  & 2 &   & 1  &   & 0.\\
\end{matrix}
\]
\end{example}

The Gelfand--Tsetlin patterns corresponding to the
tableaux counted by $K_{\lambda/\mu,\nu}(\avec,\bvec)$
are exactly the lattice points in the polytope defined by the inequalities in
\eqref{eq:gtinequalities} and the equalities in
\cref{lem:flaggedGT} together with the equalities imposed
by $\nu$.
Let us call this polytope \defin{$\mathcal{P}(\lambda/\mu,\nu,\alpha,\beta)$}.
This polytope is in general not an integral polytope
(first proved in \cite{DeLoeraMcAllister2004}) meaning that
some of its vertices are not lattice points.
Nevertheless, this polytope has a polynomial Ehrhart function
as we shall see later in \cref{cor:flaggedKostkaPolynomial}.

Similarly, for any cylindric diagram $D$
we can define the \defin{cylindric Gelfand--Tsetlin polytope}
\defin{$\mathcal{P}(D,\nu)$} where lattice points correspond to
cylindric tableaux with shape $D$ and weight $\nu = (\nu_1,\dotsc,\nu_n)$.
We can describe $D$ via some skew shape, $\lambda/\mu$ (using $m$ rows)
and the cylindric shape repeats every $\ell$ columns.
Then the equalities and inequalities used to define
$\mathcal{P}(D,\nu)$ are---in addition to those in \eqref{eq:gtinequalities}---
the following:
\begin{itemize}
 \item $x_{nj}=\lambda_j$ for all $j$;
 \item $x_{0j}=\mu_j$ for all $j$;
 \item $\nu_i = \sum_j x_{i,j}-x_{i-1,j}$ for all $i$ (total number of $i$s in the tableau is $\nu_i$);
 \item $x_{im} + \ell \geq x_{i1}$ for all $i$ (cylindric condition).
\end{itemize}

\subsection{Saturation of flagged Kostka coefficients}

Recall that $K_{\lambda/\mu,\nu}(\avec,\bvec)$
denotes the number of flagged semi-standard Young tableaux
with shape $\lambda/\mu$ and weight $\nu$ and where the
entries in each row are constrained by the flags $\avec$ and $\bvec$.

We have the following easy implication for any positive integer $k$:
\[
 K_{\lambda/\mu,\nu}(\avec,\bvec) > 0 \implies
 K_{k\lambda/k\mu,k\nu}(\avec,\bvec) > 0.
\]
To prove this, simply consider a semi-standard tableau contributing to the
first quantity. We then make $k$ copies of every column, and obtain
a tableau contributing to the second quantity.
On the Gelfand--Tsetlin side, this map simply corresponds to multiplying
all entries in the pattern by the factor $k$.

The main theorem of this section is the reverse implication.
\begin{theorem}[Saturation for flagged skew Kostka coefficients]\label{thm:flagSaturation}
Let $\lambda/\mu$ be a skew shape and $\avec,\bvec$ be flags.
Furthermore, let $k \geq 1$ be a positive integer. We then have the implication
\[
 K_{k\lambda/k\mu,k\nu}(\avec,\bvec) > 0  \implies
 K_{\lambda/\mu,\nu}(\avec,\bvec) > 0.
\]
\end{theorem}
The non-flagged version (of regular Kostka coefficients)
follows from the celebrated Knutson--Tao Saturation theorem
proved in \cite{KnutsonTau1999}.
They prove saturation for the Littlewood--Richardson coefficients,
a family which contains the skew Kostka coefficients.

The implication for the non-flagged (ordinary) skew Kostka coefficients,
\[
 K_{k\lambda/k\mu,k\nu} > 0  \implies
 K_{\lambda/\mu,\nu} > 0,
\]
is a bit easier to prove compared to the Littlewood--Richardson
coefficients. Such a proof was given by the first author in
\cite{Alexandersson2015KSaturation}.
It turns out that the method used there is in fact general enough
to prove \cref{thm:flagSaturation} without any modification.

Instead of repeating that proof in detail, we
simply give an outline of the proof
and describe how it is compatible with the flagged case.
The proof uses the language of Gelfand--Tsetlin patterns
and it has some consequences in terms of the corresponding polytope
$\mathcal{P}(\lambda/\mu,\nu,\alpha,\beta)$.
\medskip

The proof method used in \cite{Alexandersson2015KSaturation}
together with the observation regarding the flags
gives a proof of \cref{thm:flagSaturation}.
We provide an outline below.

\begin{proof}[Proof sketch]
It suffices to show that if the polytope
$\mathcal{P}(k\lambda/k\mu,k\nu,\alpha,\beta)$ contains a lattice point $G$,
we can modify $G$ via a sequence of steps (described explicitly in \cite{Alexandersson2015KSaturation})
\[
 G = G_1 \to G_2 \to \dotsb \to G_\ell = G'
\]
where every $G_j$ is a lattice point in
$\mathcal{P}(k\lambda/k\mu,k\nu,\alpha,\beta)$,
and the final GT-pattern $G'$
has the property that every entry is a multiple of $k$.
If we then divide all entries in $G'$ by $k$, we have found a
lattice point in $\mathcal{P}(\lambda/\mu,\nu,\alpha,\beta)$
and thus proved $K_{\lambda/\mu,\nu}(\avec,\bvec) > 0$.

The sequence of steps that takes $G$ to $G'$ never
touches the first and last row (the shape of the tableau).
In fact, once an entry becomes a multiple of $k$ it is
never modified in subsequent steps.

By reading the rows in a GT-pattern from left to right,
top to bottom, we obtain a string. We can in each step
ensure that we decrease this string in lexicographic order.
Moreover, we can always perform a step whenever
at least one entry in the current GT-pattern is not a multiple of $k$.

The flag $a_j = r$ for $G$ implies that
\[
 k\mu_j = x_{0,j} = x_{1,j}=\dotsb = x_{r-1,j}.
\]
That is, all these entries are equal to $k\mu_j$.
Similarly, if $b_j=s$, then
\[
 k\lambda_j = x_{n,j} = x_{n-1,j}= \dotsb = x_{s,j}.
\]
Since the above steps never modify entries which are multiples of $k$,
the flag conditions imposed by $\avec$ and $\bvec$ are preserved in each step. This last paragraph
is the adaptation needed for the flagged case.
\end{proof}
The above proof actually shows something stronger.

\begin{corollary}
Let $\mathcal{P}(\lambda/\mu,\nu,\alpha,\beta)$
be a non-empty GT-polytope.
Then the the lexicographically smallest point $P$
in the polytope is a vertex with integer entries.
\end{corollary}
\begin{proof}
The fact that $P$ is lexicographically smallest ensures that it is a vertex,
so we only need to show that it has integer coordinates.

Suppose $P$ was not an integer point, but only having rational coordinates.
We could then magnify it with some factor $k$, making $kP$
an integer vertex of $\mathcal{P}(k\lambda/k\mu,k\nu,\alpha,\beta)$.
This would then be the lexicographically smallest member of this larger polytope,
so in the above method, we would not be able to to any additional step. This implies
that all entries of $kP$ are multiples of $k$,
so $P$ must have integer coordinates.
\end{proof}

\subsection{Saturation of cylindric Kostka coefficients}

We can now use \cref{thm:flagSaturation} to prove saturation
for the cylindric skew Kostka coefficients.
Let $\defin{\CSSYT(D,\nu)}$ denote the set of cylindric SSYTs with shape $D$ and weight $\nu$.
When $k$ is a positive integer we let $kD$ denote the cylindric diagram obtained from $D$
where every box has been subdivided into $k$ boxes horizontally.
For example, with $k=3$, we have
\[
D=
\ytableaushort{
{\cdot},
{\cdot},
{\,}{\,}{\,},
{\none}{\none}{\,}{\,}{\none[\cdot]},
{\none}{\none}{\,}{\,}{\none[\cdot]}
}
\qquad
\implies
\qquad
kD=
\ytableaushort{
{\cdot}{\,}{\,},
{\cdot}{\,}{\,},
{\,}{\,}{\,}{\,}{\,}{\,}{\,}{\,}{\,},
{\none}{\none}{\none}{\none}{\none}{\none}{\,}{\,}{\,}{\,}{\,}{\,}{\none[\cdot]},
{\none}{\none}{\none}{\none}{\none}{\none}{\,}{\,}{\,}{\,}{\,}{\,}{\none[\cdot]}
}
\]

\begin{theorem}[Saturation of cylindric Kostka coefficients]\label{thm:cylindricSaturation}
Let $D$ be a cylindric diagram and $k \geq 1$ an integer.
Then
\[
  |\CSSYT(kD, k\nu)| > 0 \iff |\CSSYT( D,  \nu)| >0.
\]
\end{theorem}
\begin{proof}
As before, the $(\Longleftarrow)$ direction is easy.
Suppose now we have some tableau $T \in \CSSYT(kD, k\nu)$.
By making a cut between the ``last'' and ``first'' column of $T$ (this cut can be between any two columns of different shape)
we get a skew shape $k\lambda/k\mu$.
We may then interpret $T$ as an element in $\SSYT(k\lambda/k\mu,\avec,\bvec)$,
where the flag $\avec$ is chosen to match the first column of $T$
and the flag $\bvec$ is chosen to match the last column of $T$.

The flags $\avec$ and $\bvec$ now have the property that \emph{any}\footnote{Note that this is just an injection in general, not a bjiection!}
tableau in $\SSYT(k\lambda/k\mu,\avec,\bvec)$ can be reinterpreted
as a cylindric tableau in $\CSSYT(kD, k\nu)$.

Since $\SSYT(k\lambda/k\mu,\avec,\bvec)$ is non-empty,
\cref{thm:flagSaturation} implies that there is some tableau $T'$
in $\SSYT(\lambda/\mu,\avec,\bvec)$. The flags now ensure that we can glue the first
and last column of $T'$ to obtain a valid cylindric tableau in $\CSSYT(D, \nu)$.
\end{proof}

According to A.~Knutson (personal communication) \cref{thm:cylindricSaturation} also follows from
a result in \cite{Belkale2007}.

\section{Polynomiality of stretched row-flagged Kostka coefficients}

Our first goal is to express the Kostka coefficient $K_{\lambda/\mu,\nu}(\avec,\bvec)$ in terms of contingency tables.
Our approach is inspired by the method used in \cite{Thawinrak2022x} but contingency tables
are better suited when dealing with non-symmetric polynomials.
The non-flagged case was proved earlier by E.~Rassart~\cite{Rassart2004}.

\begin{definition}\label{def:contingency}
Let $\alpha,\beta \in \setZ^{\ell}$ and $\avec,\bvec \in \setN^\ell$.
Then let $\defin{M^{(\avec,\bvec)}_\beta[\alpha]}$ be the
number of non-negative $\ell{\times}\ell$-matrices $M = (m_{ij})$
satisfying
\begin{enumerate}
 \item column $j$ has sum $\alpha_j$ for all $1\leq j \leq \ell$,
 \item row $i$ has sum $\beta_i$ for all $1\leq i \leq \ell$,
 \item $i \notin \{a_j,\dotsc,b_j\}$ implies $m_{ij}=0$.
\end{enumerate}
\end{definition}
In other words, $M^{(\avec,\bvec)}_\beta[\alpha]$ enumerate
certain contingency tables with given margins $\alpha$ and $\beta$.
Note that enumerating contingency tables is $\#P$-hard \cite{DyerKannanMount1997}.

\begin{lemma}\label{eq:hContingency}
Let $\alpha,\avec,\bvec$ be as before. Then the expansion of $\completeH_{\alpha}(\avec,\bvec)$
in the monomial basis is given by
\[
  \completeH_{\alpha}(\avec,\bvec) = \sum_{\beta} M^{(\avec,\bvec)}_\beta[\alpha] \cdot \xvec^{\beta},
\]
where the sum is over all non-negative integer vectors $\beta$.
\end{lemma}
\begin{proof}
 Consider the left hand side, which is a product as in \eqref{eq:hprod}.
 Consider the coefficient of $\xvec^{\alpha}$. On the left hand side,
 we must choose monomials where the exponents add to $\alpha$.
 But this is exactly what the right hand side encodes:
 the choices of columns of the contingency tables (matrices) correspond to the factors in the product.
\end{proof}

\begin{example}
Let us compute the coefficient of $x_1^2 x_2^5 x_3^7$ in $\completeH_{356}(\avec,\bvec)$,
for $\avec=(1,2,2)$, $\bvec=(2,3,3)$.
By definition, $\completeH_{356}(\avec,\bvec)$ is the product
\begin{align*}
 \completeH_{356}(\avec,\bvec) &= \left(x_1^3+ x_1^2x_2 + x_1x_2^2 + x_2^3\right) \cdot \\
 &\phantom{=}\left(x_2^5+x_2^4x_3 + x_2^3x_3^2+\dotsb + x_2 x_3^4 + x_3^5 \right) \cdot \\
 &\phantom{=}\left(x_2^6+x_2^5x_3+x_2^4x_3^2 + \dotsb +x_2 x_3^5+ x_3^6 \right)
\end{align*}
and if we expand this, there are 5 ways to obtain the monomial $x_1^2 x_2^5 x_3^3$,
by choosing monomials from each factor:
\begin{align*}
  &x_1^2 x_2 \cdot  x_3^5  \cdot x_2^4 x_3^2 \; + \; x_1^2 x_2 \cdot  x_2 x_3^4  \cdot x_2^3 x_3^3 \; + \\
  &x_1^2 x_2 \cdot  x_2^2 x_3^3 \cdot x_2^2 x_3^4 \; + \;  x_1^2 x_2 \cdot  x_2^3 x_3^2 \cdot x_2 x_3^4 + \\
  &x_1^2 x_2 \cdot  x_2^4 x_3 \cdot x_3^5.
\end{align*}
The corresponding contingency tables, with column sums $3,5,6$, respectively,
and row sums $2,5,7$ are the following five:
\[
\ytableausetup{boxsize=1.0em}
\begin{ytableau}
 2 & 0 & 0 \\
 1 & 0 & 4 \\
 0 & 5 & 2 \\
\end{ytableau}
\quad
\begin{ytableau}
 2 & 0 & 0 \\
 1 & 1 & 3 \\
 0 & 4 & 3 \\
\end{ytableau}
\quad
\begin{ytableau}
 2 & 0 & 0 \\
 1 & 2 & 2 \\
 0 & 3 & 4 \\
\end{ytableau}
\quad
\begin{ytableau}
 2 & 0 & 0 \\
 1 & 3 & 1 \\
 0 & 2 & 5 \\
\end{ytableau}
\quad
\begin{ytableau}
 2 & 0 & 0 \\
 1 & 4 & 0 \\
 0 & 1 & 6 \\
\end{ytableau}
\]
\end{example}

\begin{proposition}\label{prop:kostkaInContingency}
Let $\lambda/\mu$ be a skew shape with $\ell$ rows and $\avec,\bvec \in \setN^\ell$.
Then
 \begin{equation}
 K_{\lambda/\mu,\beta}(\avec,\bvec) =
 [\xvec^{\beta}]\schurS_{\lambda/\mu}(\avec,\bvec) =
 \sum_{\sigma \in \symS_n}
	\varepsilon(\sigma) M^{(\sigma(\avec),\bvec)}_\beta[\sigma(\lambda + \delta) - (\mu+\delta)],
\end{equation}
where $\delta \coloneqq (n-1,n-2,\dotsc,2,1,0)$.
\end{proposition}
\begin{proof}
We can restate the Jacobi--Trudi identity \eqref{eq:JT} as
\begin{equation}
 \schurS_{\lambda/\mu}(\avec,\bvec) =
 \left| h_{(\lambda_i + \delta_i) - (\mu_j+\delta_j) }(a_j,b_i) \right|_{1 \leq i,j \leq \ell}.
\end{equation}
The determinant is then expressed as a sum over
permutations and we get
\begin{equation}
 \schurS_{\lambda/\mu}(\avec,\bvec) =
 \sum_{\sigma \in \symS_n}
	\varepsilon(\sigma)
  h_{\sigma(\lambda + \delta) - (\mu+\delta)}(\sigma(\avec),\bvec).
\end{equation}
Expanding both sides in the monomial basis and using \cref{eq:hContingency},
we finally have
\begin{equation}
 [\xvec^{\beta}] \schurS_{\lambda/\mu}(\avec,\bvec) =
 \sum_{\sigma \in \symS_n}
	\varepsilon(\sigma)
	M^{(\sigma(\avec),\bvec)}_\beta[\sigma(\lambda + \delta) - (\mu+\delta)].
\end{equation}
\end{proof}

\subsection{Polynomiality of contingency table counting}

Let $\lambda,\mu$ be integer partitions with
at most $\ell$ parts and such that $\lambda/\mu$ is a skew shape.
Moreover, let $\nu, \rho, \beta, \eta \in \setZ^\ell$.
Our next goal is to show that the map
$k \mapsto K_{k\lambda+\nu/k\mu+\rho,k\beta+\eta}(\avec,\bvec)$
is a polynomial in $k$ whenever $k$ is sufficiently large.
By \cref{prop:kostkaInContingency},
it then suffices to show that the map
\[
  k \mapsto M^{(\avec,\bvec)}_{k\beta+\eta}[k\alpha+\gamma]
\]
is a polynomial for $k \gg 0$, where $\alpha, \beta \in \setN^\ell$
and $\gamma,\eta \in \setZ^\ell$.

In order to show this, we need some facts about contingency tables.
Let $\uvec = (u_1,\dotsc,u_m)$ and $\vvec = (v_1,\dotsc,v_n)$,
and consider the non-negative $m{\times}n$ integer matrices $\xvec = (x_{ij})$
with row sums given by $\uvec$ and column sums given by $\vvec$.
This can be encoded as a system of equations,
\[
  A\xvec = \begin{pmatrix} \uvec \\ \vvec \end{pmatrix}, \qquad x_{ij} \geq 0.
\]
The matrix $A$ is \emph{totally unimodular} (see \cite{DahmenMicchelli1988,Mount1995}) that is,
every minor has determinant in $\{-1,0,1\}$.
Moreover, if we add one or more equations of the form $x_{ij}=0$
as in \cref{def:contingency},
it is easy to see that the corresponding matrix $A$ is still totally unimodular;
adding a row of the form $(0,\dotsc,0,1,0,\dotsc,0)$ to the bottom of a matrix
preserves the property of being unimodular\footnote{Consider the cofactor expansion.}.

For a general matrix $A$, the \defin{partition function of $A$} is defined as
the function
\[
  \defin{\phi_A(y_1,\dotsc,y_m)} \coloneqq \left| \left\{ \xvec \in \setN^s : A \xvec = \yvec \right\} \right|
\]
which simply count the number of non-negative integer solutions to $A \xvec = \yvec$.
In particular, for fixed $\avec$, $\bvec$, $\alpha,\beta,\gamma$, there is
a unimodular matrix $A$ (depending on $\avec$ and $\bvec$) such that
\begin{equation}\label{eq:contingencyAsPartitionFunc}
 M^{(\avec,\bvec)}_{k\beta+\eta}[k\alpha+\gamma] = \phi_A(k\alpha+\gamma, k\beta+\eta).
\end{equation}
\medskip

The following definitions and theorem are due to Sturmfels~\cite{Sturmfels1995}.
Given a $d{\times}n$-matrix $A$ with columns $(a_1,\dotsc,a_n)$,
let $\mathrm{pos}(A)$ be the cone $\{\sum_{j=1}^n a_j y_j \in \setR^n : y_1,\dotsc,y_n \geq 0 \}$.
A subset $S \subseteq \{1,2,\dotsc,n\}$ is a \defin{basis} if the columns indexed by $S$
is a submatrix (denoted $A_S$) of $A$ with same rank as $A$.
The \defin{chamber complex} is the the polyhedral subdivision of the cone $\mathrm{pos}(A)$,
given by the common refinement of all cones $\mathrm{pos}(A_S)$ as $S$ ranges over all bases.
A \defin{chamber} of the chamber complex is a cell with maximum dimension.

The following theorem is a special case of a more general theorem by Sturmfels, see \cite[Thm. 1]{Sturmfels1995}.
\begin{theorem}[Sturmfels, 1995]
Suppose the $d{\times}n$-matrix $A$ is a totally unimodular matrix, and let $\phi_A$ be its partition function.
Then for each chamber $C$, there is a polynomial $P(y_1,\dotsc,y_n)$,
so that $\phi_A(y_1,\dotsc,y_n) = P(y_1,\dotsc,y_n)$.
\end{theorem}
Hence, for $k$ sufficiently large and $A$ unimodular, $\phi_A(k\alpha+\gamma, k\beta+\eta)$
is a polynomial in $k$, since the ray $k\alpha+\gamma, k\beta+\eta$ eventually remains
in the same chamber as $k$ grows (or it leaves the cone $\mathrm{pos}(A)$ entirely,
in which case we get the constant zero polynomial).

\begin{corollary}\label{cor:flaggedKostkaPolynomial}
The map $K_{k\lambda+\nu/k\mu+\rho,k\beta+\eta}(\avec,\bvec)$
is polynomial in $k$ for $k$ sufficiently large.
Furthermore, the map $k \mapsto K_{k\lambda/k\mu,k\beta}(\avec,\bvec)$ is a polynomial in $k$
whenever $k\geq 0$.
\end{corollary}
\begin{proof}
The first statement follows directly from combining \cref{prop:kostkaInContingency}
with \eqref{eq:contingencyAsPartitionFunc}.

For the second statement, we need to use the fact that $K_{k\lambda/k\mu,k\beta}(\avec,\bvec)$
is a quasipolynomial in $k$ since it is the Ehrhart polynomial of face of a Gelfand--Tsetlin polytope.
These two facts together gives polynomiality for all non-negative $k$.
\end{proof}

\begin{example}
Note that the map $k \mapsto M^{(\avec,\bvec)}_{k\beta}[k\alpha +\gamma]$ in \eqref{eq:contingencyAsPartitionFunc}
is indeed only a polynomial for sufficiently large $k$.
For $\avec=(1,1)$, $\bvec=(2,2)$, $\alpha =(2,0)$, $\beta=(1,1)$ and $\gamma =(-4,4)$,
the values of $M^{(\avec,\bvec)}_{k\beta}[k\alpha +\gamma]$ as $k=1,2,3,4,\dotsc$
are $0,1,3,5,5,\dotsc$.
The five contingency tables for $k \geq 4$ are of the form
\[
\ytableausetup{boxsize=2.5em}
 \begin{ytableau}
 \scriptstyle{k-j}  & \scriptstyle{j} \\
 \scriptstyle{k-4+j} & \scriptstyle{4-j} \\
\end{ytableau}
\qquad j=0,1,2,3,4, \quad k\geq 4.
\]
\end{example}

\begin{theorem}[B. Kostant, see also \cite{Shrivastava2022x}]
Let $\lambda,\mu,\nu$ be partitions of length at most $\ell$,
such that $|\nu|=|\lambda|+|\mu|$. Then
\begin{equation}
c^{\nu}_{\lambda \mu} = \sum_{\sigma \in \symS_\ell} \varepsilon(\sigma) K_{\mu, \sigma(\nu+\rho)-(\lambda+\rho)},
\end{equation}
where $\rho = \tfrac12 (\ell-1,\ell-3,\dotsc,1-\ell)$.
\end{theorem}
Combining this with the formula above, we get the following expression for
the Littlewood--Richardson coefficients in terms of contingency table counts:
\begin{equation}
c^{\nu}_{\lambda \mu} =
\sum_{\substack{\sigma \in \symS_\ell \\ \tau \in \symS_\ell}}
\varepsilon(\sigma \tau)
M_{\sigma(\nu+\rho)-(\lambda+\rho)}[\tau(\mu+\delta)-\delta].
\end{equation}
With the same argument as above, we see that the map $k \mapsto c^{k\nu}_{k\lambda,k\mu}$
is a polynomial in $k$ for large $k.$
It then follows from a result by Berenstein and Zelevinsky
that this map is a quasipolynomial, since it is the Ehrhart function
of a rational polytope. These facts together shows that the map is
a polynomial for all $k \geq 0$.

\begin{remark}
It is tempting to see if one can show polynomiality in a
setting which generalizes the flag condition.
One can define a Schur-like polynomials where we sum over semistandard Young tableaux
of shape $\lambda/\mu$, but entries in row $j$ must be in some fixed set $S_j$.
For example, take $\lambda/\mu = (4,3)/(2)$, $S_1 = \{1,3,4,5\}$, $S_2 = \{1,2,3,4,5\}$.
Then the number of such SSYT of shape $k \lambda / k \mu$ and weight $k(2,1,1,1)$
is given by
\[
  3, 7, 12, 19, 27, 37, 48,61, \dotsc, \text{ for } k=1,2,\dotsc,
\]
which is only a quasipolynomial. If it were a polynomial, it must be bounded by
the corresponding classical Kostka coefficients, which in this case are given
by the $3$rd degree polynomial $k \mapsto \frac{1}{3} \left(2 k^3+6 k^2+7 k+3\right)$,
but the data above does not fit a polynomial of degree at most $3$.

The reason that the above method fails is that there is no clear analog of the Jacobi--Trudi identity for
these more generalized conditions on row entries.
\end{remark}

\subsection{Polynomiality for cylindric Kostka coefficients}

Let $D$ be a cylindric diagram and $\schurS_D(x_1,\dotsc,x_n)$
be the cylindric Schur function restricted to $n$ variables.
Let us designate one of the columns in the diagram to be the first column.
We can then cut the diagram immediately to the left of the first
column and unfold the diagram to obtain some skew shape
$\lambda/\mu$.
Any semi-standard tableau of shape $\lambda/\mu$ is also
a proper cylindric semi-standard tableau if the last
column of the SSYT is compatible with the first column.
Since columns are increasing, a column is
uniquely determined by its entries.
Suppose $\nu_1,\nu_2,\dotsc,\dotsc,\nu_m \in \{0,1\}^n$
are indicator vectors for the possible contents in the first column
of $\lambda/\mu$.
We can then express the cylindric Schur polynomial
as a sum over flagged Schur polynomials as follows,
where each term corresponds to a fixed filling of the first column:
\begin{equation}\label{eq:cylindricAsFlagSum}
 \schurS_{D}(x_1,\dotsc,x_n) =
 \sum_{i=1}^m
 \xvec^{\nu_i} \cdot \schurS_{\lambda/(\mu+1^\ell)}(\avec_i,\bvec_i).
\end{equation}
The shape $\lambda/(\mu+1^\ell)$ is the diagram with the first
column removed, $\avec_i$ is the flag that ensures
that entries in row $i$ are at least as large as the $i$th
entry in the first column, and $\bvec_i$
is the flag that ensures that if we wrap around
the last column of $\lambda/\mu$, the entries are cylindrically
compatible with the fixed first column.

\begin{example}
Suppose $D$ is the diagram below and we
want to compute $\schurS_{D}(x_1,\dotsc,x_5)$
using the formula in \eqref{eq:cylindricAsFlagSum}.
\[
\ytableausetup{boxsize=1em}
D=
\ytableaushort{
{\cdot},
{\cdot},
{\,}{\,}{\,},
{\none}{\none}{\,}{\,}{\none[\cdot]},
{\none}{\none}{\,}{\,}{\none[\cdot]}
}
\]
There are $\binom{5}{3}=10$ choices of entries
in the first (leftmost) column.
If we for example fill the first column with the entries in $\{1,2,3\}$ we get
\[
\ytableaushort{
{3},
{2},
{1}{\,}{\,},
{\none}{\none}{\,}{\,}{\none[\leq]}{\none[3]},
{\none}{\none}{\,}{\,}{\none[\leq]}{\none[2]}
}
\]
The 10 choices leads to the 10 flagged Schur functions with flags as indicated below
where $\{a<b<c\} \subset \{1,2,3,4,5\}$:
\[
\ytableaushort{
{\none[a]}{\none[\leq]}{\,}{\,},
{\none}{\none}{\none}{\,}{\,}{\none[\leq]}{\none[c]},
{\none}{\none}{\none}{\,}{\,}{\none[\leq]}{\none[b]}
}
\]
To be precise, the flags for each choice of first column are $\avec = (a,1,1)$ and $\bvec = (b,c,5)$,
as we put the value $1$ in $\avec$ for rows unaffected by the choice of enties in the first column of $D$,
and similarly for $\bvec$.
\end{example}

\begin{corollary}\label{cor:cylindricKostkaPolynomial}
The cylindric Kostka coefficients can
be expressed as the following sum of flagged skew Schur Kostka coefficients:
\[
 [\xvec^\beta] \schurS_{D} = \sum_{i=1}^m
 K_{\lambda/(\mu+1^\ell),\beta - \nu}(\avec_i,\bvec_i).
\]
Moreover, the map $k \mapsto [\xvec^{k\beta}] \schurS_{k D}$
is a polynomial in $k$.
\end{corollary}
\begin{proof}
Note that for any $k \geq 1$, we have that
\[
 [\xvec^{k\beta}] \schurS_{kD} = \sum_{i=1}^m
 K_{k\lambda/(k\mu+1^\ell),k\beta - \nu_i}(\avec_i,\bvec_i),
\]
so the number of terms in the sum remains constant as $k$ increases.
We know from \cref{cor:flaggedKostkaPolynomial}
that each term in the sum is polynomial in $k$ for $k \gg 0$.
The fact that we also can model $[\xvec^{k\beta}] \schurS_{kD}$
as the Ehrhart quasipolynomial of the
rational polytope $\mathcal{P}(D,\beta)$
defined earlier implies that the map must be a polynomial for all $k \geq 0$.
\end{proof}

\begin{conjecture}[Cylindric extension of the King--Tollu--Toumazet conjecture, \cite{KingTolluToumazet2004}]
For any cylindric shape $D$ and weight $\beta$, the map
\[
 k \mapsto |\CSSYT(k D, k \beta)|
\]
is a polynomial in $\setN[k]$.
Note that this is the Ehrhart polynomial of the
cylindric Gelfand--Tsetlin polytope $\mathcal{P}(D,\beta)$.
\end{conjecture}

\begin{example}
Let $D$ be the cylindric shape
as in \cref{ex:cylindricGT} and let $\beta = 2211$
as before. The sequence $|\CSSYT(k D, k \beta)|$ for $k=0,1,2,\dotsc$ is then $1, 3, 6, 10, 15, 21,\dotsc$
which is the polynomial $k \mapsto \frac12(k+2)(k+1)$.
\end{example}

\section{Appendix}

\subsection{Proof of Proposition~\ref{prop:stackedRibbon}}

The goal is to show that for a given cylindric diagram $D$ with strict edges $E$,
\begin{equation}\label{eq:mnruleAppendix}
 \sum_{S\subseteq E}(-1)^{|S|}m_S(P) = 
 \begin{cases}
  (-1)^{\height(D)}  &\text{ if $D$ is a non-pure stacked ribbon,} \\
  \width(D) \cdot (-1)^{\height(D)}  &\text{ if $D$ is a pure stacked ribbon,} \\
  0  &\text{ otherwise.} \\
 \end{cases}
\end{equation}

As described before, we decompose $D$ into ribbons $D =   L_1 \cup L_2 \cup \dotsb \cup L_\ell \cup F$,
where the $L_j$ are loop ribbons and $F$ is either a loop ribbon or a skew shape.
Denote the elements in the minimal equivalence class
determined by $S\subseteq E$ is denoted by \defin{$M_S$} (if this minimum is unique).
We first cover the case $\ell=0$ where we only need to consider a diagram $F$.

\medskip
\noindent
\textbf{Case when $\ell=0$:}
The diagram $F$ can be either disconnected, contain a $2{\times}2$-box,
be a ribbon or be a loop ribbon.

\begin{itemize}
 \item If $F$ is disconnected, then the left hand side of \eqref{eq:mnruleAppendix} is always zero.
 
 \item If $F$ contains a $2{\times}2$-box $\ytableaushort{cd,ab}$,
 then $b$ and $d$ are never in $M_S$. Hence,
 \[
   \sum_{\substack{S\subseteq E \\ bd \in S}}(-1)^{|S|}m_S(P) \text{ and }
   \sum_{\substack{S\subseteq E \\ bd \notin S}}(-1)^{|S|}m_S(P)
 \]
are equal but with opposite signs, so again \eqref{eq:mnruleAppendix} is 0.

\item If $F$ is a ribbon then the first box in
the leftmost column is a minimum
so this must be in every $M_S$ contributing to the sum. 
In order to have a unique minimum, $S$ must include all edges in $E$
which are not in the first column of $F$.
For example, all edges between boxes marked $\circ$ must be in $S$:
\[
F \quad = \quad
 \ytableaushort{
	{\diamond},
	{\diamond},
	{\diamond},
{\ast}\;{\circ},
	{\none}{\none}{\circ},
	{\none}{\none}{\circ},
	{\none}{\none}{\circ}\;{\circ},
    {\none}{\none}{\none}{\none}{\circ}\;\;
	}
\]
If the penultimate box marked $\diamond$ in the first column is \emph{not}
in $M_S$, we can toggle the edge between the two last boxes. 
Hence, we may sum only over those $S$ in \eqref{eq:mnruleAppendix},
where the penultimate box in the first column \emph{is} a member in $M_S$
(and edges in all other columns must be in $S$).
This leads to only two terms left with sum $(-1)^{\height(F)}$.

\item If $F$ is a loop ribbon, let us consider the
possible ways for $S$ to produce a unique minimum.
If $S=E$, then all boxes are part of $M_S$.
Otherwise, $M_S$ consists of the first few boxes in some column
containing at least two boxes whose strict edges are not all present in $F$
and the remaining strict edges in all other columns must be present in $S$.
In other words, we only need to consider $S$ where the strict edges in $E\setminus S$
are in the same column.
Consider now the highest
$e \in E\setminus S$.
If by adding $e$ to $S$ does not affect the minimum,
we have a sign-reversing involution between the terms from
$S$ and $S \cup \{e\}$.
What remains are the cases where adding $e$ alters the minimum
and this can only happen when $E\setminus S = \{e\}$.
In such a case, note that $m_S$ is the number of
strict edges in that column and that $|S| = |E|-1$.

Adding these for all columns gives a total of
\begin{align*}
(-1)^{|E|}\cdot |F| + (-1)^{|E|-1}\cdot |E| &= (-1)^{|E|}(|F|-|E|) \\
&= (-1)^{\height(F)}\cdot \width(F),
\end{align*}
where the last identity is due to \eqref{eq:loopWidthRel}.
\end{itemize}

What remains is now the case where $\ell \geq 1$.
We will show via a series of sign-reversing involutions that there are
only a few selections $S \subseteq E$ that we need to consider.
Each such involution is of the form above,
where we \emph{toggle} the membership of some strict edge in $S$,
in such a way that $M_S$ is unchanged. 
We start at the bottom of the diagram and work our way up.

\medskip
\noindent
\textbf{Case when $\ell \geq 1$:}
In the illustrations, the edges included in $S$ are shown in black,
the edges not in $S$ are shown in white,
and edges where a choice has not yet been made (or are
part of a sign-reversing involution argument)
are shown in black and white.

\medskip
\noindent
\emph{Restriction 1---All strict edges in $L_1$ in $S$:}

Let us start from the bottom of the shape.
Since $F$ is non-empty, there is at least one
strict edge $\{a,b\} \in E$
such that $a \in L_1$ but $b \notin L_1$ (see R1.a).
If $a \notin M_S$, $\{a,b\}$ can be toggled
so we may assume $a \in M_S$. Note that as $a$ can not be a minimal box on $L_1$ as the box $b$ above it is not on $L_1$. That means, the only way it can be in $M_S$ is for all the boxes in $L_1$ to be in $M_S$. This implies that all strict edges in $L_1$ must be be in $S$ (see R1.b).

\medskip
\noindent
\emph{Restriction 2---All strict edges connecting $L_1$
to non-corners of $L_2$ (or $F$) are not in $S$:}

Take an edge $e$ connecting a box in $L_1$ to
a non-corner box in $L_2$.
The edge $e'$ to the left of $e$, also connects
a box in $L_1$ to a box in $L_2$ (see R2.a below).
If $e \in S$, then we can toggle $e'$.
Hence, we may assume $e \notin S$.
The same argument applies to any such $e$ (R2.b).

\begin{figure}[!ht]
\centering
\begin{align*}
&
\includegraphics[page=1,width=0.2\linewidth]{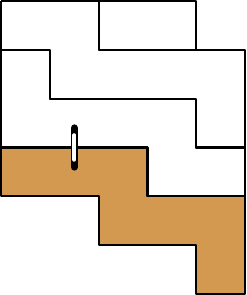}
&&
\includegraphics[page=2,width=0.2\linewidth]{toggleProofPics}
&&
\includegraphics[page=3,width=0.2\linewidth]{toggleProofPics}
&&
\includegraphics[page=4,width=0.2\linewidth]{toggleProofPics} 
\\
&
\text{R1.a} &&
\text{R1.b} &&
\text{R2.a} &&
\text{R2.b} \\
\end{align*}
\end{figure}

\medskip
\noindent
\emph{Restriction 3---For all $i\leq \ell$, all strict edges in $L_i$ are in $S$,
and from strict edges connecting $L_i$ and $L_{i-1}$ exactly one is in $S$:}

We will assume $\ell$ is at least $2$ and show that these assumptions do
not change the total for $i=2$. The rest follows inductively.
As $F\leq \varnothing$, as in R1.a, we can find a box directly above
a non-minimal box in $L_2$. Note that the edge $e_B$ connecting that
to a box $B$ in $L_2$ can be toggled unless $B$ is part of the unique minimum.
Let $C\neq B$ be the minimal box in $L_2$ on the same row as $B$,
and let $e_C$ be the edge connecting $C$ to a box in $L_2$.
If adding or removing $e_C$ in a configuration does not change
the unique minimum, $e_C$ can be toggled.
So we will restrict our attention to configurations where
adding/deleting $e_C$ \emph{changes} the unique minimum.
In these configurations, when $e_C$ is out, $B$ can not be
in the unique minimum as $B\in m_S(P)$ implies $C\in m_S(P)$.
Therefore we can assume $e_C$ is in.

Note that at this point, by our restrictions, the configurations $S$ we
are considering have $e_C \in S$,  $C,B \in m_S(S)$ and if
we set ${S_0} \coloneqq  S\setminus \{e_C\}$ then $C\notin m_S({S_0})$.
These restrictions do not directly force all strict edges in $L_2$ to be in $S$.
Because our shape wraps around, configurations like the one shown in $R3.b$
can also have  $B\notin m_S({S_0})$.
Fortunately for our calculations, these type of cases cancel out because
of the following subtle argument: Assume there is some strict edge in $L_2$
that is not in $S$ and let $D$ be the minimal box of $L_2$ in that column,
connected to $L_1$ with an edge $e_D$. Then we can toggle $e_D$ without
breaking our assumptions. In particular, the edge $e_C$ changes the minimum,
whether $e_D$ is in or not---see how adding or removing $e_C$ affects
minimum with $e_D$ included in (3.c) and (3.d) below.

\begin{figure}[!ht]
\centering
\begin{align*}
&
\includegraphics[page=5,width=0.2\linewidth]{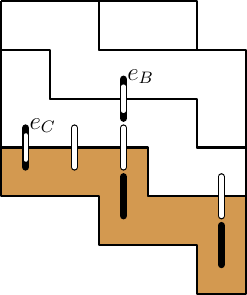}
&&
\includegraphics[page=6,width=0.2\linewidth]{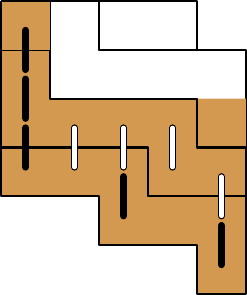}
&&
\includegraphics[page=7,width=0.2\linewidth]{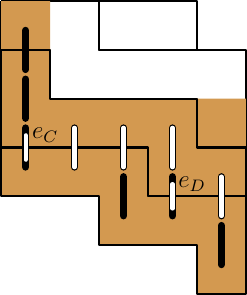}
&&
\includegraphics[page=8,width=0.2\linewidth]{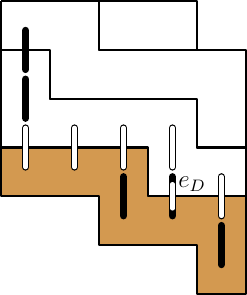}
\\
&
\text{R3.a} &&
\text{R3.b} &&
\text{R3.c} &&
\text{R3.d} \\
\end{align*}
\end{figure}

That means we can indeed assume all strict edges on $L_2$ are included
in $S$ (see R3.e).
As a consequence, all boxes in $L_2$ are in the same equivalence
class and any strict edge other than $e_C$ between $L_1$ and $L_2$ being
included breaks our assumption that $e_C$ changes the minimum.
So we can assume $e_C$ is the only edge between $L_1$ and $L_2$
that is included, and all others are out (see R3.f).

\medskip
\noindent
\emph{Restriction 4---Strict edges connecting $L_\ell$ to non-minimal
boxes of $F$ are not in $S$:}

We can also use the same argument as we used for Restriction 2 (See R4).
\begin{figure}[!ht]
\centering
\begin{align*}
&
\includegraphics[page=9,width=0.2\linewidth]{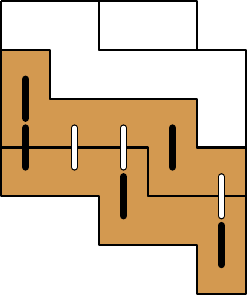}
&&
\includegraphics[page=10,width=0.2\linewidth]{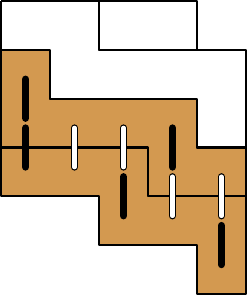}
&&
\includegraphics[page=11,width=0.2\linewidth]{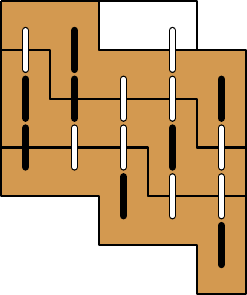}
\\
&
\text{R3.e} &&
\text{R3.f} &&
\text{R4} \\
\end{align*}
\end{figure}

\textbf{Different cases for $F$:}
Note that we only have the strict edges connecting
the minimal boxes of $F$ to $L_\ell$
and the edges on $F$ left to consider, all other edges are either included or
omitted in $S$ per our restrictions so far.
Let $h,w$ denote the height and width of a loop ribbon, respectively.
Our arrangement so far uses a total of $\ell \cdot (h+1)-1$ vertical edges,
and all $M \coloneqq \ell(w+h)$ boxes of the $\ell$ loop ribbons are
included in the unique minimum.
We will finish our proof by looking at the possible
shapes for $F$ as we have done in the case $\ell=0$.

\begin{itemize}
\item
Let us first consider the case $F$ is a ribbon.
Note that we already know that the edges connecting non-minimal
boxes of $F$ to $L_\ell$ are out (See R5.a).
If $F$ consists of a single row, there is only one edge to consider,
so the total value we get is
\[
(-1)^{(\ell (h+1)-1)}M+(-1)^{\ell(h+1 )}(M+1)=(-1)^{(\ell \cdot h+ \ell)}.
\]
Otherwise, consider the rightmost minimal box on $F$.
Let $e$ be the bottommost strict edge on its column (See R5.b).
We can toggle $e$ unless the higher box it is connected to is a minimum,
which can only happen if all other previously unset edges on its column are in.
Consider the next minimum to the right on $F$, if another exists,
and let $f$ be the edge connecting it to $L_\ell$ (See R5.c).
If $e$ is in, $f$ can be toggled to give us $0$ so we can assume $e$ is out.
We can indeed continue with this argument, and assume, for all but the
leftmost minimum on $F$, the edge connecting it to $L_\ell$ and
all but the highest edge on its column are in, and the highest is out.
For the leftmost minimum, we can do the same assumption, except that
we do not know whether the highest edge on its column is in or out
(Note that this edge would be the one connected to $L_\ell$ if
the top-left box of $F$ is a minimum).
We can calculate our final sum by toggling that edge (shown marked in R5.d):

\begin{align*}
&(-1)^{(\ell (h+1)+\height(F)-1}(M+\height(F)-1)+(-1)^{(\ell \cdot (h+1)+\height(F))}(M+\height(F))\\
&=(-1)^{(\ell \cdot(h+1)+\height(F))}=(-1)^{\height(D)}.
\end{align*}

\begin{figure}[!ht]
\centering
\begin{align*}
&
\includegraphics[page=12,width=0.2\textwidth]{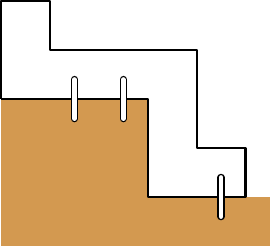}
&&
\includegraphics[page=13,width=0.2\textwidth]{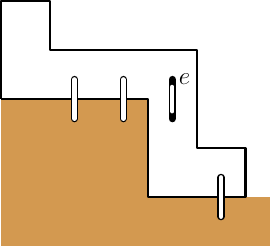}
&&
\includegraphics[page=14,width=0.2\textwidth]{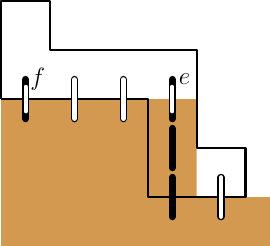}
&&
\includegraphics[page=15,width=0.2\textwidth]{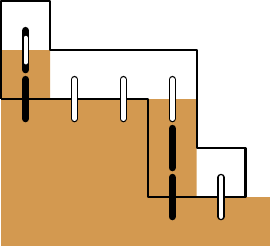}
\\
&
\text{R5.a} &&
\text{R5.b} &&
\text{R5.c} && 
\text{R5.d} \\
\end{align*}
\caption{The various stages for $F$.}\label{fig:R5Pics}
\end{figure}

 \item If $F$ is the union of two or more disconnected ribbons,
 we can apply the same argument.
 As the final edges are independent of each other,
 we get a total of $0$ by inclusion exclusion.
 
 \item If $F$ contains a $2{\times}2$-box $\ytableaushort{ab,cd}$, we can pick
 such a box where $c$ is a minimal box of $F$.
 For the configurations where $d$ is not part of the unique minimum,
 we can toggle the edge between $b$ and $d$ to end up with $0$.
 That means, we can add the assumption that $d$ is part of the unique minimum.
 In this case though, the  strict edge connecting $c$ to $L_l$ has no effect,
 and it can be toggled.
 That means if such a box exist, the total contribution of all configurations is $0$.

\item If $F$ is a loop ribbon, then let $e_1,e_2,\dotsc,e_h$ be the vertical edges
that lie in $F$ in order, starting from any edge of our choosing.
If we assume $e_1$ is out, then we can argue, as in the non-loop ribbon
case, that the total contribution is $(-1)^{(\ell \cdot (h+1)+h-1)}$.
Then we assume $e_1$ is in an $e_2$ is out, and get the same contribution
$(-1)^{(\ell \cdot (h+1)+h-1)}$.
Continuing along and looking at $e_1,\ldots,e_{i-1}$ in and $e_i$ out case
for each $i$ gives us $h\cdot (-1)^{(\ell \cdot h+ \ell+h-1)}$ contribution in total.
We are left with the case where all $h$ the vertical edges of $F$ are in.
In that case, let $e_0$ be one of the edges connecting a minimal box
of $F$ to $L_\ell$. In the configurations where $e_0$ does not alter the minimum,
it can be toggled, so we can limit our attention to the cases where $e_0$ alters the minimum.
That can only happen when all other edges connecting $F$ to $L_\ell$ are out.
Toggling $e_0$ and adding up the above cases gives us our total:
\[
 h\cdot (-1)^{(\ell \cdot(h+1)+h-1)}+(-1)^{(\ell \cdot (h+1)+h-1)}M+(-1)^{(\ell(h+1)+h)}(M+w+h).
\]
This simplifies to
\begin{align*}
(-1)^{(\ell \cdot(h+1)+h)}( w+h-h) &=(-1)^{(\ell\cdot (h+1)+h)}\width(F) \\
&=(-1)^{(\height(F)}\width(D).
\end{align*}
\end{itemize}

\bibliographystyle{alphaurl}
\bibliography{bibliography}

\end{document}